\documentclass[11pt]{article}
\usepackage{amsmath, amssymb, theorem, latexsym, epsfig, graphics, eufrak}
\usepackage{subfigure}
\numberwithin{equation}{section}

\theoremstyle{plain}
\theorembodyfont{\itshape}
\newtheorem{theorem}{Theorem}[section]
\newtheorem{proposition}[theorem]{Proposition}
\newtheorem{lemma}[theorem]{Lemma}
\newtheorem{corollary}[theorem]{Corollary}
                                                                                
\theorembodyfont{\rmfamily}
\newtheorem{definition}[theorem]{Definition}

\newtheorem{example}[theorem]{Example}
\newtheorem{remark}[theorem]{Remark}
\newtheorem{convention}[theorem]{Convention}
\newtheorem{problem}[theorem]{Problem}
\newtheorem{conjecture}[theorem]{Conjecture}
\newenvironment{proof}{{\noindent \textbf{Proof}\,\,}}{\hspace*{\fill}$\Box$\medskip}

\title{On monodromy  eigenfunctions of Heun equations and boundaries of phase-lock areas in a model of overdamped Josephson effect}
\author{V.M.Buchstaber\thanks{ Permanent address: Steklov Mathematical Institute, 8, Gubkina street, 119991, Moscow, Russia.  Email: buchstab@mi.ras.ru}
\thanks{All-Russian Scientific Research Institute for Physical and Radio-Technical Measurements (VNIIFTRI),}
\thanks{Supported by part by RFBR grant 14-01-00506.}, 
A.A.Glutsyuk\thanks{ CNRS, France (UMR 5669 (UMPA, ENS de Lyon) and UMI 2615 (Lab. J.-V.Poncelet)), Lyon, France. 
Email:
aglutsyu@ens-lyon.fr}, 
\thanks{National Research University Higher School of Economics (HSE), Moscow, Russia}
 \thanks{Supported by part by RFBR grants 13-01-00969-a, 16-01-00748, 16-01-00766 
 and ANR grant ANR-13-JS01-0010.}}

\begin{document}
\maketitle
\def\la{\lambda}
\def\cc{\mathbb C}
\def\oc{\overline\cc}
\def\rr{\mathbb R}
\def\zz{\mathbb Z}
\def\cp{\mathbb{CP}}
\def\diag{\operatorname{diag}}
\def\mcl{\mathcal L}
\def\nn{\mathbb N}
\def\var{\varepsilon}
\def\mcr{\mathcal R}

\begin{abstract} We study a family of double confluent Heun equations of the form $\mcl E=0$, where $\mcl=\mcl_{\la,\mu,n}$ 
is a family of differential operators of order two  acting on germs of holomorphic functions in one complex variable. They depend on 
complex parameters $\la$, $\mu$, $n$. Its restriction to real parameters satisfying the inequality $\la+\mu^2>0$ 
is a linearization of the family of nonlinear equations on two-torus 
modeling the Josephson effect in superconductivity.  
We show that 
for every $b,n\in\cc$ satisfying a certain  ``non-resonance condition'' and every parameter values $\la,\mu\in\cc$, $\mu\neq0$  
there exists an entire function $f_{\pm}:\cc\to\cc$ (unique up to  constant factor) 
such that $z^{-b}\mcl(z^b f_{\pm}(z^{\pm1}))=d_{0\pm}+d_{1\pm}z$ for some 
$d_{0\pm},d_{1\pm}\in\cc$. The constants $d_{j,\pm}$ are expressed as functions of the parameters. 
This result has several applications. First of all, it gives the description of those  values $\la$, $\mu$, $n$, $b$ 
for which the monodromy operator of the corresponding Heun equation has  eigenvalue $e^{2\pi i b}$. It also gives the description 
of those values  $\la$, $\mu$, $n$ for which the monodromy is parabolic: has multiple eigenvalue. We consider  the rotation 
number $\rho$ of the dynamical system on two-torus as a function of parameters restricted to a 
surface $\la+\mu^2=const$. 
The phase-lock areas are its level sets having non-empty interiors. For general families of dynamical systems 
the problem to describe the boundaries of the phase-lock areas is known to be very complicated. 
In the present paper we include the results in this direction that were obtained by methods of complex variables. 
In our case the phase-lock areas 
exist only for integer rotation numbers 
 (quantization effect), and the complement to them is an open set.   On their complement the rotation number function is an 
 analytic submersion that induces its fibration by analytic curves.
The above-mentioned result on parabolic 
monodromy implies the explicit description of the union of boundaries of the phase-lock areas as solutions of an explicit transcendental 
functional equation. 
For every $\theta\notin\zz$ we get a  description of the set $\{\rho\equiv\pm\theta(mod2\zz)\}$.
\end{abstract} 

\tableofcontents

\section{Introduction: statement of results, sketch of proof and plan of the paper}

We study the problem to find those 
solutions of {\it non-homogeneous} double-confluent Heun equations that are monodromy eigenfunctions. 
Our study is motivated by applications 
to nonlinear equations modeling the Josephson effect in superconductivity. The main results, the existence and uniqueness of the 
above solutions 
(Theorems \ref{noncom} and \ref{xi=0}) are stated in Subsection 1.1. Applications to monodromy eigenfunctions and eigenvalues of homogeneous 
double confluent Heun equations and to nonlinear equations modeling Josephson effect 
are presented in Subsections 1.1 and 1.3, Sections 4  and 5. 

Each eigenfunction is the product of a monomial $z^b$ and a function $f(z)$ holomorphic on $\cc^*$. The  Heun equation  is equivalent to recurrence relations on the Laurent coefficients of the function $f$. The proofs of the above-mentioned results are based on studying 
the latter recurrence relations. We prove  existence and uniqueness Theorem \ref{cons} for converging solutions of a more general 
class of recurrence relations (stated in Subsection 1.2 and proved in Section 2). Its proof is based on ideas from hyperbolic 
dynamics and a fixed point argument for appropriate contracting mapping.

\subsection{Main result: existence and uniqueness of monodromy eigenfunctions of non-homogeneous double confluent Heun equations} 
We consider the family of double confluent Heun equations 
\begin{equation} \mcl E=z^2E''+(nz+\mu(1-z^2))E'+(\la-\mu nz)E=0; \ \ \ n,\la,\mu\in\cc, \ \mu\neq0.\label{heun}\end{equation}
They are well-known linear differential equations, see \cite[formula (3.1.15)]{sl} that have singular points only 
at zero and infinity, both of them are irregular. 
Our goal is to study existence of the eigenfunctions of their monodromy operators (see \cite[chapter 7, subsection 3.2]{arnil} and 
the definition in Subsection 4.1) 
with a given eigenvalue $e^{2\pi i b}$, $b\in\cc$: the latter 
functions are solutions of equation (\ref{heun}) having the form 
\begin{equation} E(z)=z^bf(z), \ f(z) \text{ is holomorphic on } \cc^*.\label{multi}\end{equation}
The converging Laurent series of the function $f(z)$ is split into two parts, $f(z)=f_+(z)+f_-(z^{-1})$, where $f_{\pm}$ are holomorphic functions on $\cc$ and $f_+(0)=0$. These functions satisfy  non-homogeneous equations of the type 
\begin{equation} z^{-b}\mcl(z^b f_{\pm}(z^{\pm1}))=d_{0\pm}+d_{1\pm}z\label{nonhom} \end{equation}
One of our main results is the following.

\begin{theorem} \label{noncom} For every $(n,\la,\mu,b)\in U$, 
\begin{equation} U=\{ (n,\la,\mu,b)\in\cc^4 \ | \ \mu\neq0, \  b, b+n\notin\zz\},\label{defu}\end{equation}
there exist holomorphic  functions $f_{\pm}(z)$ on a neighborhood of zero, $f_+(0)=0$ such that the  functions 
$f_{\pm}(z^{\pm1})$  satisfy equations (\ref{nonhom}) for 
appropriate $d_{0\pm}(n, \la, \mu,b)$, $d_{1\pm}(n,\la,\mu,b)$. The functions $f_{\pm}$ are unique 
up to constant factors (depending on the parameters),  and they are 
 entire functions: holomorphic on $\cc$. For every sign index $\pm$ the corresponding vector $(d_{0\pm},d_{1\pm})$ is uniquely defined 
 up to scalar factor depending on parameters. The above constant factors can be chosen so that both $f_{\pm}$ and $d_{j\pm}$ depend 
 holomorphically on $(n,\la,\mu,b)\in U$ and $f_{\pm}(z)$ are real-valued in $z\in\rr$ for real parameter values. 
\end{theorem}

\begin{corollary} \label{cordon} Let $(n,\la,\mu,b)\in U$. The corresponding equation (\ref{heun})   has a monodromy eigenfunction with eigenvalue $e^{2\pi i b}$, $b\in\cc$, 
if and only if the corresponding vectors 
$d_{\pm}=(d_{0\pm},d_{1\pm})$ are proportional: 
\begin{equation} d_{0+}d_{1-}-d_{0-}d_{1+}=0.\label{schivka}\end{equation}
\end{corollary}

Theorem \ref{noncom} will be proved in  the next subsection and Section 2. Corollary \ref{cordon} will be proved in the next subsection. 
The explicit formulas for the functions $f_{\pm}$ and 
$d_{j\pm}$ will be given in Section 3. Equivalent versions of equation (\ref{schivka}) as explicit functional equations in 
parameters obtained in a more direct way 
will be presented in Section 4. 

\begin{theorem} \label{xi=0} For every $n\in\cc\setminus\zz_{\leq0}$ and $(\la,\mu)\in\cc^2$, $\mu\neq0$     there exists a unique function 
$E(z)\not\equiv0$ (up to constant factor) holomorphic on a neighborhood of zero such that $\mcl E=c$, where $c$ is independent on 
$z$. The function $E$ 
is entire and can be normalized so that $c=\xi_{n-1}(\la,\mu)$ depends holomorphically on $(n,\la,\mu)$ and 
$E$ also depends holomorphically on $(n,\la,\mu)$.
\end{theorem}

Theorem \ref{xi=0} will be proved in the next subsection. 

\begin{remark} Theorem \ref{xi=0} is closely related to the question of the existence of a solution holomorphic at 0 
of equation (\ref{heun}) (such a solution is automatically entire, i.e., holomorphic on $\cc$). This question was studied 
by V.M.Buchstaber and S.I.Tertychnyi in 
\cite{bt1}. The existence of a solution $E$  from Theorem \ref{xi=0} and explicit expressions  for  $E$ and the corresponding 
function $\xi_{n-1}(\la,\mu)$ (analytic in $(\la,\mu)\in\cc^2$) were given  in \cite[pp. 337--338]{bt1}. (This was done for $n\in\nn$, but 
 these results remain valid for all $n\in\cc\setminus\zz_{\leq0}$.)  
The existence result implies that if $\xi_{n-1}(\la,\mu)=0$, then 
the homogeneous equation (\ref{heun}), i.e., $\mcl E=0$ has a solution holomorphic on $\cc$. 
A conjecture stated by V.M.Buchstaber and S.I.Tertychnyi  in loc. cit. (under the additional assumption that 
$n\in\nn$, which, in fact, can be omitted) said that the converse is true: 
if equation $(\ref{heun})$ has a holomorphic solution at 0, then $\xi_{n-1}(\la,\mu)=0$. This conjecture was studied for 
$n\in\nn$ in loc. cit. 
 and \cite{bt2}, where it was reduced to a series of conjectures on polynomial solutions of auxiliary Heun equations and 
 non-vanishing of determinants of 
 modified Bessel functions of the first kind. All these conjectures were solved in \cite{bg}. As the next corollary shows, 
 Theorem \ref{xi=0} implies the conjecture 
 of Buchstaber and Tertychnyi immediately for all $n\in\cc\setminus\zz_{\leq0}$. 
 \end{remark}
 
 \begin{remark} The method used in \cite{bt1, bt2, bg} was to consider 
 a ``conjugated'' family of Heun equations, for which equations having polynomial solutions  were described by an explicit algebraic equation on parameters in \cite{bt0}.  The proof of 
 the above-mentioned Buchstaber--Tertychnyi conjecture was obtained in \cite{bg} from a solution of their 
 problem about polynomial solutions. The method of the present paper allows to prove the conjecture directly, without using the 
 conjugated family and polynomial solutions. A possibility to prove the conjecture without using the condition of non-vanishing 
 of the determinants  of  modified Bessel functions  was mentioned in \cite{bt1} at the place, where the conjecture 
 was stated. We realize this possibility in the present paper. 
As is explained in \cite{bg}, 
positivity of Bessel determinants proved there is a result interesting itself. 
 \end{remark}
 
 \begin{corollary} \label{cxi} (cf. \cite[theorem 3.5]{bg}) For every $n\in\cc\setminus\zz_{\leq0}$ and $\mu\neq0$   equation 
 (\ref{heun}) has an entire solution, if and only if $\xi_{n-1}(\la,\mu)=0$, where $\xi_{n-1}(\la,\mu)$ is the 
 function from Theorem \ref{xi=0}  introduced  in \cite[formula (31), p. 337]{bt1}, 
 see also formula (\ref{xil}) in Subsection 4.3 below.
 \end{corollary}
 
 \begin{proof} Let $\xi_{n-1}(\la,\mu)=0$. Then the function $E$ from Theorem \ref{xi=0} is an entire solution 
 of equation (\ref{heun}): $\mcl E=0$.  Conversely, let equation (\ref{heun}) have a solution $E$ 
 holomorphic at 0.  If $\xi_{n-1}(\la,\mu)\neq0$, then there exists a holomorphic function 
 $\hat E$ on a neighborhood of zero satisfying the equation $\mcl \hat E=\xi_{n-1}(\la,\mu)\neq0$, by Theorem \ref{xi=0}. 
 This together with the uniqueness 
 statement of Theorem \ref{xi=0} implies that $\hat E= E$ up to constant factor, hence $\mcl\hat E=0$. The contradiction thus obtained proves the corollary. 
\end{proof}

\subsection{Solutions of three-term recurrence relations}
Let us look for {\it formal} solutions of equation (\ref{heun}) of type (\ref{multi}). 
Equation (\ref{heun}) is equivalent to  the recurrence relations 
\begin{equation}((k+b)(k+b+n-1)+\la)a_k-\mu(k+b+n-1)a_{k-1}+\mu(k+b+1)a_{k+1}=0,\label{recur}\end{equation}
which can be written in the matrix form 

$$\left(\begin{matrix} a_{k}\\  a_{k+1}\end{matrix}\right)=A_k\left(\begin{matrix}  a_{k-1}\\  a_k
\end{matrix}\right),$$
\begin{equation}A_k=\frac{k+b+n-1}{k+b+1}\left(\begin{matrix}  0 & \frac{k+b+1}{k+b+n-1}\\ 
 1 & -\frac{\la+(k+b)(k+b+n-1)}{\mu(k+b+n-1)}\end{matrix}\right).\label{mat1}\end{equation}

\begin{remark} \label{remnoncom} 
A function $f_+(z)=\sum_{k\geq1}a_kz^k$ satisfies equation (\ref{nonhom}) for some $d_{j+}$, if and only if its Taylor coefficients $a_k$ 
satisfy (\ref{recur}), or equivalently, (\ref{mat1}) for $k\geq2$. Similarly, a function $f_-(z^{-1})=\sum_{k\leq0}a_kz^{k}$ 
satisfies (\ref{nonhom}), if and only if its coefficients satisfy (\ref{recur}) for $k\leq-1$. 
\end{remark}

\begin{proof} {\bf of Corollary \ref{cordon}.} Let $E(z)=z^b\sum_{k\in\zz}a_kz^k$ be a solution of equation (\ref{heun}) having type (\ref{multi}). Then 
\begin{equation} E(z)=z^b(f_+(z)+f_-(z^{-1})), \ \ f_+(z)=\sum_{k\geq1}a_kz^k, \ f_-(z)=\sum_{k\geq0}a_{-k}z^k.
\label{defef}\end{equation}
The coefficients $a_k$ satisfy (\ref{recur}) for all $k$. This together with the above remark implies that the functions 
$f_{\pm}(z^{\pm1})$ satisfy (\ref{nonhom}). The corresponding expressions $d_{\pm}=d_{0\pm}+d_{1\pm}z$ should cancel out, since 
$E$ is a solution of the homogeneous equation. This implies (\ref{schivka}). Conversely, let $f_{\pm}(z^{\pm1})$ be solutions of 
(\ref{nonhom}), and let (\ref{schivka}) hold: that is, the vectors $d_+$ and $d_-$ are proportional. 
 Then we can normalize the latter vectors, and hence, the  corresponding solutions by constant factors (not both vanishing simultaneously) so that 
$d_++d_-=0$. Then the function $E$ given by (\ref{defef}) is a solution of equation (\ref{heun}). 
\end{proof} 

As it is shown below, Theorem \ref{noncom} is implied by the following general theorem 

\begin{theorem} \label{cons}  Consider recurrence relations 
\begin{equation} f_ka_{k-1}+g_ka_k+h_ka_{k+1}=0\label{rec}\end{equation}
in sequence $a_k$, $k\in\zz_{\geq0}$, where sequences $g_k, f_k, h_k\in\cc$  satisfy the following conditions:   
\begin{equation} f_k,h_k\neq0,  \text{ for every } k,\label{cond3}\end{equation}

%
\begin{equation} f_k,h_k=o(g_k), \text{ as } k\to\infty.\label{condi}\end{equation}
Then there exists a unique series $\sum_{k=0}^{+\infty} a_kz^k\not\equiv0$ (up to constant factor) 
with $a_k$ satisfying  (\ref{rec}) for $k\geq1$ 
and having non-zero convergence radius. Moreover, this series  
converges on all of $\cc$. 
\end{theorem}

{\bf Addendum to Theorem \ref{cons}.} {\it Let in the conditions of Theorem \ref{cons} the coefficients $f_k$, $g_k$, $h_k$ 
depend holomorphically on   parameters that represent a point of a finite-dimensional complex manifold. Let 
asymptotics (\ref{condi}) hold uniformly on compact 
subsets  in the parameter manifold. Then the function $\sum_{k=0}^{+\infty}a_kz^k$ can be normalized to 
depend meromorphically on the parameters. In the case, when the parameter manifold is Stein and contractible\footnote{The condition of contractibility  may be weakened to the condition of vanishing of the second cohomology group with integer coefficients.  
The latter condition implies in particular that the quotient of the fundamental group by its commutant 
has trivial torsion. Vanishing of the second cohomology 
 together with Stein property  is needed to guarantee that each analytic hypersurface in the parameter manifold 
is the zero locus of a holomorphic function \cite[chapter VII, section B, proposition 13]{gr}.}, 
 this function can be normalized to be holomorphic in the parameters.}

\medskip
Theorem \ref{cons} and its addendum will be proved in the next section.

\begin{remark} \label{run} In the series  $\sum_ka_kz^k$ from Theorem  \ref{cons} for every $k\geq0$ the two neighbor coefficients 
$a_k$, $a_{k+1}$ 
do not vanish simultaneously: hence, they present a point $(a_k:a_{k+1})\in\cp^1$. Indeed, 
each pair of neighbor coefficients $(a_k,a_{k+1})$ 
determines a unique sequence satisfying (\ref{rec}). This  follows from the fact that for every $k\geq1$ the coefficient 
$a_{k\pm1}$ is expressed as a linear combination of $a_{k\mp1}$ and $a_k$ by (\ref{rec}), since $f_k,h_k\neq0$. 
Hence, if some two 
neighbor coefficients $a_{k-1}, a_k$ vanish, then  all the coefficients vanish, and the series is zero, -- a contradiction. 
\end{remark}

\begin{theorem} \label{conv} Let $b,n\in\cc$, 
%
Then for  every $k_0\in\zz$ such that 
\begin{equation} k+b+n-1,\ k+b+1\neq0 \text{ for every } k>k_0,\label{kko}\end{equation}
 for every $\la,\mu\in\cc$, $\mu\neq0$  there exists and unique  nonzero 
 one-sided series $\sum_{k\geq k_0}a_kz^k$  (up to constant factor) converging on some punctured disk centered at 0 
that satisfies recurrence relations (\ref{recur}) (or equivalently, (\ref{mat1})) for $k> k_0$. (In what follows this solution of relations (\ref{recur}) is called the {\bf forward solution.})
Similarly, for every $k_0\in\zz$ such that 
\begin{equation} k+b+n-1,\ k+b+1\neq0 \text{ for every } k< k_0\label{kko2}\end{equation}
  there exists and unique  
one-sides series $\sum_{k\leq k_0}a_kz^k$ (up to multiplicative constant) that satisfies recurrence relations (\ref{recur}) for 
$k<k_0$ and converges outside some disk 
centered at 0.  In what follows this solution of relations (\ref{recur}) is called the {\bf backward solution.})  
Both series converge  on the whole punctured complex line $\cc^*$. They can be normalized to depend 
holomorphically on those parameters $(n,\la,\mu, b)$ for which inequality (\ref{kko}) (respectively, (\ref{kko2})) 
holds for the given number $k_0$. 
\end{theorem}

\begin{example} \label{exkko} Let in the conditions of Theorem \ref{conv} one have  $b, n+b\notin\zz$ (cf. (\ref{defu})). 
Then its statements hold  for all $k_0\in\zz$, since 
inequalities (\ref{kko}) hold for all $k\in\zz$. Otherwise, if either $b\in\zz$, or $b+n\in\zz$, then the statements of Theorem \ref{conv} 
\begin{equation} \text{ hold for } k> k_0 \text{ whenever } k_0\geq\max\{ m \in \{ -1-b, 1-b-n\} \ | \ m\in\zz\}\label{kko+}
\end{equation}
\begin{equation} \text{ hold for } k< k_0 \text{ whenever } k_0\leq\min\{ m \in \{ -1-b, 1-b-n\} \ | \ m\in\zz\}\label{kko-}
\end{equation}
\end{example} 
Theorem \ref{conv} together with Remark \ref{remnoncom} and the first statement of Example \ref{exkko} 
imply Theorem \ref{noncom}. 
%

\begin{proof} {\bf of Theorems \ref{conv} and \ref{noncom} (existence and uniqueness of solutions).} The coefficients 
$$f_k=-\mu(k+b+n-1), \ g_k=(k+b)(k+b+n-1)+\la, \  h_k=\mu(k+b+1)$$
of recurrence relations (\ref{recur}) satisfy the conditions of Theorem \ref{cons} for $k>k_0$ ($k<k_0$). Indeed, the asymptotics (\ref{condi}) is obvious. Inequalities $f_k,h_k\neq0$ follow from (\ref{kko}) (respectively, (\ref{kko2})).
This together with Theorem \ref{cons} implies the statement of  Theorem \ref{conv}, and hence, Theorem \ref{noncom} on existence and 
uniqueness of solutions. The local 
holomorphicity in the parameters follows from the addendum to Theorem \ref{cons}. The global holomorphicity will be proved 
later on, in Subsection 3.5. 
\end{proof} 

\begin{proof} {\bf of Theorem \ref{xi=0}.} Let $b=0$, $n\in\cc\setminus\zz_{\leq0}$. Then inequalities (\ref{kko}) hold for 
 $k>k_0=0$. Therefore, there exists a unique series $E(z)=\sum_{k=0}^{+\infty}a_kz^k$ converging on a 
neighborhood of the origin, whose coefficients satisfy (\ref{recur}) for $k\geq1$, and it converges on all of $\cc$ (Theorem \ref{conv}). 
The system of relations (\ref{recur}) for $k\geq1$ is equivalent to the statement that $\mcl E=const$. This proves Theorem \ref{xi=0}. 
Holomorphicity on the parameters follows from the analogous statement of Theorem \ref{conv}. 
\end{proof}

\subsection{Historical remarks, applications and plan of the paper}

Our results are motivated by applications to the family 
 \begin{equation}\frac{d\phi}{dt}=-\sin \phi + B + A \cos\omega t, \ A,\omega>0, \ B\geq0.\label{josbeg}\end{equation}
  of  nonlinear equations, which arises in several models in physics, mechanics and geometry: in a model  of the 
Josephson junction  in superconductivity (our main motivation), see \cite{josephson, stewart, mcc, bar, schmidt}; 
 in  planimeters, see  \cite{Foote, foott}. 
Here $\omega$ is a fixed constant, and $(B,A)$ are the parameters. Set 
$$\tau=\omega t, \ l=\frac B\omega, \ \mu=\frac A{2\omega}.$$
The variable change $t\mapsto \tau$ transforms (\ref{josbeg}) to a 
non-autonomous ordinary differential equation on the two-torus $\mathbb T^2=S^1\times S^1$ with coordinates 
$(\phi,\tau)\in\rr^2\slash2\pi\zz^2$: 
\begin{equation} \dot \phi=\frac{d\phi}{d\tau}=-\frac{\sin \phi}{\omega} + l + 2\mu \cos \tau.\label{jostor}\end{equation}
The graphs of its solutions are the orbits of the vector field 
\begin{equation}\begin{cases} & \dot\phi=-\frac{\sin \phi}{\omega} + l + 2\mu \cos \tau\\
& \dot \tau=1\end{cases}\label{josvec}\end{equation}
on $\mathbb T^2$. The {\it rotation number} of its flow, see \cite[p. 104]{arn},  is a function $\rho(B,A)=\rho(B,A;\omega)$ 
of the parameters of the vector field. It is given by the formula
$$\rho(B,A;\omega)=\lim_{k\to+\infty}\frac{\phi(2\pi k)}k,$$
where $\phi(\tau)$ is an arbitrary solution of equation (\ref{jostor}). 

The  {\it phase-lock areas} are the level subsets of the rotation number in the $(B,A)$-plane 
with non-empty interior. They have been studied 
by V.M.Buchstaber, O.V.Karpov, S.I.Tertychnyi et al, see \cite{bktje}--\cite{bg}, \cite{4} and references therein.
Each phase-lock area is an infinite chain of adjacent domains separated by points called the 
{\it adjacency points} (or briefly, {\it adjacencies}). The description of their coordinates 
as solutions of analytic functional equations was conjecturally stated by V.M.Bushstaber and S.I.Tertychnyi in  \cite{bt1} and 
proved by the authors of the present paper in \cite{bg}. 
Namely, the family of non-linear equations  was reduced in \cite{bt0, tert} to the two following 
families  of second order linear differential equations of double confluent Heun type: equation (\ref{heun}) with
$$n=l+1, \ \la=\frac1{4\omega^2}-\mu^2$$
 and the equation   
\begin{equation} \mcl E=z^2E''+((-l+1)z+\mu(1-z^2))E'+(\la+\mu(l-1)z)E=0.\label{heun2*}\end{equation}
Equation (\ref{heun2*}) is obtained from equation (\ref{heun}) via the substitution $l=1-n$. 

\begin{remark} Heun equations (\ref{heun}) and (\ref{heun2*}) corresponding to  the family (\ref{josvec}) of dynamical systems 
on torus are those corresponding to real parameters $n$, $\omega$, $\mu$, and thus, real $\la$. In the present paper we 
are studying general Heun equation (\ref{heun}) with arbitrary complex parameters $n$, $\la$, $\mu$.
\end{remark}

It was shown in \cite{4} that $l=\frac B{\omega}\in\zz$ at all the adjacencies. 
In the case, when $l\geq0$, Buchstaber and Tertychnyi 
have  shown that the adjacencies correspond exactly to those parameter values, 
for which $l$ is integer and  equation (\ref{heun})  has a non-trivial holomorphic solution at 0 (which is automatically an entire solution: 
holomorphic on $\cc$); see the statement in \cite[p.332, paragraph 2]{bt1} and the proof in \cite[theorem 3.3 and subsection 3.2]{bg}. 
They have explicitly constructed a family of 
holomorphic solutions for parameters satisfying an explicit functional equation $\xi_l(\la,\mu)=0$, see Corollary \ref{cxi}. They have conjectured that the latter functional 
equation describes the adjacencies completely. They have reduced this conjecture to another one saying that if equation (\ref{heun2*}) 
has a polynomial solution (which may happen only for $l\in\nn$), then equation (\ref{heun}) does not have an entire solution. Later they have shown that the second conjecture follows from the 
third one saying that appropriate determinants formed by modified Bessel functions of the first type do not 
vanish on the positive semiaxis.  This third conjecture together with the two previous conjectures 
were proved in \cite{bg}. The statement of the above-mentioned 
conjecture of Buchstaber and Tertychnyi on functional equation describing the adjacencies follows from Corollary \ref{cxi} 
and  their correspondence to  entire solutions of Heun equations. 

V.M.Buchstaber and S.I.Tertychnyi have constructed symmetries of double confluent Heun equation (\ref{heun}) \cite{bt1, bt3}. 
The symmetry $\#:E(z)\mapsto 2\omega z^{-l-1}(E'(z^{-1})-\mu E(z^{-1}))$, which is an involution of its solution space, was constructed 
in \cite[equations (32), (34)]{tert2}. It corresponds to the symmetry $(\phi,t)\mapsto(\pi -\phi,-t)$ of the nonlinear equation (\ref{josbeg}); 
the latter symmetry was found in \cite{RK}. In \cite{bt3} they have found new nontrivial symmetries in the case, when $l\in\zz$ and 
equation 
(\ref{heun2*}) does not have polynomial solutions. 


\begin{convention} Everywhere in the paper by {\it formal} solution $(a_k)_{k\geq k_0}$ (or $(a_k)_{k\leq k_0}$)  of linear recurrence 
relation $f_ka_{k-1}+g_ka_k+h_ka_{k+1}=0$ we mean a (one- or two-sided) sequence of complex numbers $a_k$ satisfying the relation 
for all $k>k_0$ (respectively,  $k<k_0$). (Here one may have two-sided infinite sequences.) 
If in addition, the power series $\sum_ka_kz^k$ converges on some annulus centered at 0 (for all the relations under consideration, this would automatically imply convergence on all of $\cc^*$) then the formal solution under question is called simply {\it a solution}: the adjective ``converging'' is omitted for simplicity. 
\end{convention}
In Section 3 we write down explicit formulas for  solutions of recurrence relations (\ref{recur}) using the proof of 
Theorem \ref{cons}. Then in Section 4 we deduce  explicit  functional equations satisfied by monodromy eigenvalues of double confluent Heun equations (explicit versions of Corollary \ref{cordon}). 

In Section 5 we apply  results of Sections 3 and 4 to phase-lock areas in the model of Josephson effect. 

 \begin{remark} The problem to describe the boundaries of the phase-lock areas for the considered system was studied in 
  \cite{bt0, bt1, bg}. Special points of the boundaries (adjacencies and points corresponding to equations  
  (\ref{heun2*}) with polynomial solutions) were described in \cite{bt1} and \cite{bt0} respectively. 
  In the present paper the union of boundaries is described by an explicit transcendental analytic equation 
  (Corollary \ref{cboun} in Subsection 5.4). 
  It is known that the ratio of the monodromy eigenvalues of the corresponding 
equation (\ref{heun}) equals $e^{\pm2\pi i\rho(A,B)}$ and their product 
 equals $e^{-2\pi i l}$.  The union of  boundaries  coincides with the set where 
 the monodromy has multiple eigenvalue and   is described by the condition that the monodromy of equation (\ref{heun}) has 
eigenvalue $\pm e^{-\pi i l}$. We get a similar description of non-integer level curves of the rotation number function. Namely, 
for $\theta\notin\zz$ the above relation between monodromy eigenvalues and the rotation number
together with the results of Sections 3, 4 imply an explicit functional equation satisfied by 
 the set $\{\rho\equiv\pm\theta(mod2\zz)\}$ (Theorem \ref{rhonon} in Subsection 5.3). 
%
  \end{remark}
Open problems on phase-lock areas and  possible approaches to some of them 
using the above description of boundaries are discussed in Subsections 5.5--5.8.  

The following new result will be also proved in Section 5 using results of Section 4. 

\begin{theorem} \label{poteig0} Let $\omega>0$, $(B,A)\in\rr^2$, $B,A>0$, $l=\frac B\omega$, $\mu=\frac A{2\omega}$, 
$\la=\frac1{4\omega^2}-\mu^2$, $\rho=\rho(B,A)$. The double confluent Heun equation (\ref{heun2*}) corresponding to the above  
$\la$, $\mu$ and $l$ has a polynomial solution, if and only if $l,\rho\in\zz$, $\rho\equiv l(mod 2\zz)$, $0\leq\rho\leq l$, 
the point $(B,A)$ lies in the boundary of a phase-lock area  and is not an adjacency. In other terms, 
the points $(B,A)\in\rr_+^2$ corresponding to equations (\ref{heun2*}) with polynomial solutions lie in boundaries of 
phase-lock areas and  are exactly their intersection points with the lines $l=\frac B\omega\equiv\rho(mod 2\zz)$, $0\leq\rho\leq l$ 
that are not  adjacencies. 
\end{theorem}

\begin{remark} V.M.Buchstaber and S.I.Tertychnyi have shown in \cite{bt0}
that if a point $(B,A)\in\rr_+^2$ corresponds to equation (\ref{heun2*}) with a polynomial solution, then $l$, $\rho$ are integers, 
$0\leq\rho\leq l$ and $\rho\equiv l(mod 2\zz)$. 
\end{remark}

\subsection{A sketch of proof of Theorem \ref{cons}.} 

For every initial condition $(a_0,a_1)$ there exists a unique sequence $(a_k)_{k\geq0}$ satisfying recurrence relations (\ref{rec}), 
by Remark \ref{run}. But in general, the series $\sum_k a_kz^k$ may diverge. We have to prove that it converges for appropriately chosen unknown initial condition. 
To do this, we use the following trick: we run the recursion in the opposite direction,  ``from infinity to zero''. That is,  take a 
big $k$ and a given ``final  condition'' $q_k=(a_k,a_{k+1})$. Then the inverse recursion gives all $a_j=a_j(q_k)$, 
$0\leq j\leq k$. It appears that the initial condition $(a_0,a_1)$ we are looking for can be obtained as a limit of the initial conditions 
$(a_0(q_k),a_1(q_k))$ obtained by the above inverse recursion (after rescaling), as $k\to\infty$. The latter holds for appropriate choice of 
the final vector $q_k$: it should be appropriately normalized by scalar factor and its 
projectivization $[q_k]=(a_k:a_{k+1})\in\cp^1$ should avoid some small explicitly specified ``bad region'', which contracts to the point 
$(0:1)$, as $k\to\infty$. 

The projectivized inverse recursion 
$$P_k:\cp^1\to\cp^1: \ [q_k]=(a_{k}:a_{k+1})\mapsto [q_{k-1}]=(a_{k-1}:a_k)$$ 
 defined by (\ref{rec}) can be considered  as the dynamical system 
$$T:(\nn_{\geq2}\cup\{\infty\})\times\cp^1\to(\nn\cup\{\infty\})\times\cp^1, \ \nn_{\geq 2}=\nn\cap[2,+\infty),$$
where for every $x\in\cp^1$ and $k\in\nn_{\geq2}$ one has 
$$ T:(k,x)\mapsto(k-1,P_k(x)); \  \ T:\infty\times\cp^1\mapsto\infty\times(1:0).$$
It appears that for every $k$ large enough $P_k$ has a strongly attracting fixed point tending to $(1:0)$ and a strongly repelling fixed point 
tending to $(0:1)$, as $k\to\infty$. This together with the ideas from basic theory of hyperbolic dynamics implies that the 
fixed point 
$p_{\infty}=\infty\times(1:0)$ of the transformation $T$ should have a unique unstable manifold: an invariant sequence $(k,[q_k])$ converging to $p_{\infty}$.  We show that  a solution $(a_k)$ of recurrence relations (\ref{rec}) gives a converging Taylor series 
$\sum_ka_kz^k$ on some neighborhood of zero, if and only if  $(a_k:a_{k+1})=[q_k]$ for all $k$, and then the series converge 
everywhere. This will prove Theorem \ref{cons}.

The existence and uniqueness of the above-mentioned unstable manifold is implied by the following discrete 
analogue of  the classical Hadamard--Perron Theorem on 
 the unstable manifold of a dynamical system at a hyperbolic fixed point.

\begin{theorem} \label{metric} Let $E_1,E_2,\dots$ be a sequence of complete metric spaces with uniformly bounded diameters. For brevity, the distance on each of them will be denoted 
$d$. Let $P_k:E_k\to E_{k-1}$ be a sequence of uniformly contracting mappings: 
there exists a  $\la$, $0<\la<1$ such that $d(P_k(x),P_k(y))<\la d(x,y)$ for every $x,y\in E_k$ and $k\geq2$. Then there exists a unique 
sequence of points 
$x_k\in E_k$ such that $x_{k-1}=P_k(x_k)$ for all $k\geq2$. One has 
\begin{equation} x_{k-1}=\lim_{m\to\infty}P_k\circ\dots\circ P_m(x),\label{convxk}\end{equation}
and the convergence is uniform in $x$: for every $\var>0$ there exists some $l\in\nn$ such that for every $m\geq l$ and every 
$x\in E_m$ one has $d(P_k\circ\dots\circ P_m(x),x_{k-1})<\var$. If in addition the spaces $E_k$ coincide with one and the same 
space $E$ and the fixed points of the mappings $P_k$ tend to some $x_{\infty}\in E$, as $k\to\infty$, then 
\begin{equation} \lim_{k\to\infty}x_k=x_{\infty}.\label{limxk}\end{equation}
\end{theorem}

\begin{proof} The proof repeats the argument of the classical proof of Hadamard--Perron Theorem. 
Consider the space $S$ of all sequences $X=(x_k)_{k\in\nn}$, $x_k\in E_k$, equipped with the distance 
$$D(X,Y)=\sup_kd(x_k,y_k).$$
The transformation 
$$T:S\to S, \ (x_1,x_2,\dots)\mapsto(P_2(x_2),P_3(x_3),\dots)$$
is a contraction. Therefore, it has a unique fixed point, which is exactly the sequence we are looking for. The second statement of the 
theorem on the uniform convergence of compositions to $x_{k-1}$ follows from the uniform convergence of iterations of the 
contracting map $T$ to its fixed point.  In the last condition of Theorem \ref{metric} statement (\ref{limxk}) follows by the above fixed point 
argument in the subspace  in $S$ of the sequences 
$(x_k)$ tending to $x_{\infty}$, as $k\to\infty$: this is a complete $T$-invariant  metric subspace in $S$, and hence, $T$ has a fixed 
point there, which coincides with the previous sequence $(x_k)$ by uniqueness. Theorem \ref{metric} is proved.
\end{proof} 

\section{Proof of Theorem \ref{cons} and its addendum} 
\def\La{\Lambda}

Recurrence relations (\ref{rec}) can be written in the matrix form 
\begin{equation} \left(\begin{matrix}  a_{k}\\  a_{k+1}\end{matrix}\right)=\La_k\left(\begin{matrix}  a_{k-1}\\  a_k
\end{matrix}\right), \ \La_k=
h_k^{-1}\left(\begin{matrix}  0 & h_k\\ 
 -f_k & -g_k\end{matrix}\right).\label{mat4}\end{equation}
Consider the inverse matrices 
\begin{equation}\La_k^{-1}=\left(\begin{matrix}  -\frac{g_k}{f_k} & -\frac{h_k}{f_k}\\  1 & 0\end{matrix}\right)\label{lak}\end{equation}
 and their projectivizations $P_k:\cp^1\to\cp^1$ acting on the projective line $\cp^1=\oc$ with 
homogeneous coordinates $(z_1:z_2)$. In  the affine chart $\{z_1\neq0\}\subset\cp^1$ with the 
coordinate $w=\frac{z_2}{z_1}$ for every $C>0$ we denote 
$$D_C=\{ |w|<C\}\subset\cc\subset\oc.$$ 
%

\begin{proposition} \label{pkconv} The transformations $P_k$ converge to the constant mapping $\cp^1\mapsto\{ w=0\}$ 
uniformly on every closed  disk $\overline D_C$, 
$C>0$, as $k\to\infty$. 
Their inverses converge to the constant mapping $\cp^1\mapsto\{ w=\infty\}$ uniformly on the complement of every disk $D_C$. 
\end{proposition}
\begin{proof} The image of a vector $(1,w)$ with $|w|\leq C$ under the matrix $\La_k^{-1}$ is the vector 
$$(u_k(w),v_k(w))=-\frac{g_k}{f_k}(1+\frac{h_k}{g_k}w, -\frac{f_k}{g_k}).$$
 Recall that $\frac{h_k}{g_k}, \frac{f_k}{g_k}\to0$, see (\ref{condi}), hence, 
$g_k\neq0$ for all $k$ large enough. The latter asymptotics and formula together 
imply that $\frac{v_k(w)}{u_k(w)}\to0$ uniformly on $\overline D_C$ 
and prove the first statement of the proposition. Let us prove its second statement. For every fixed $C>0$ and every 
$k$ large enough (dependently on $C$)  one has 
$P_k^{-1}(\cp^1\setminus D_C)\subset \cp^1\setminus D_C$, by the first statement of the proposition.  
The image of a vector $(1,w)$  under the matrix $\La_k$ is 
$(s_k(w),t_k(w))=(w,-\frac{g_kw+f_k}{h_k})$. This together with (\ref{condi}) implies that $\frac{s_k(w)}{t_k(w)}\to0$ uniformly on 
$|w|\geq C$, as $k\to\infty$; or equivalently, $P_k\to\infty$ uniformly on $\cp^1\setminus D_C$. The proposition is proved. 
\end{proof} 

\begin{proof} {\bf of  Theorem \ref{cons}.} Let $C>1$,  $E_C$ denote  the closed disk $\overline D_C\subset\cp^1$ 
equipped with the Euclidean  distance. 
There exist a $\la$, $0<\la<1$ and a $N=N(\la,C)\in\nn$ such that for every $k>N$ one has $P_k(E_C)\subset E_C$ and 
the mapping   $P_k:E_C\to E_C$ is a $\la$-contraction: 
$|P_k(x)-P_k(y)|<\la|x-y|$. This follows from the first statement of the proposition and the fact that uniform convergence of holomorphic 
functions implies uniform convergence of their derivatives on compact sets (Cauchy bound): here uniform convergence of the mappings 
$P_k$ to the constant mapping on $D_{C'}$ with $C'>C$ implies uniform convergence of their derivatives to zero on $\overline D_C$. 
The fixed point of the mapping $P_k|_{E_C}$ tends to 0, as $k\to\infty$, by uniform convergence (Proposition \ref{pkconv}). This together with 
Theorem \ref{metric} implies that there exists a unique sequence $(x_k)_{k\geq N}$ such that $P_k(x_k)=x_{k-1}$ for 
all $k>N$ and $|x_k|\leq C$. The latter sequence corresponds to a unique sequence $(a_k)_{k\geq N}$ 
(up to multiplicative constant) such that 
$x_k=(a_k:a_{k+1})$; one has 
$|w(x_k)|=|\frac{a_{k+1}}{a_k}|\leq C$ for every $k\geq N$. The sequence $(a_k)$ satisfies relations (\ref{rec}) 
for $k>N$, which are equivalent to the equalities $P_k(x_k)=x_{k-1}$. It extends to a unique sequence $(a_k)_{k\geq0}$ satisfying 
(\ref{rec}) for $k\geq1$, as in Remark \ref{run}. 
In  addition, $x_k\to0$,  i.e., $\frac{a_{k+1}}{a_k}\to0$, as 
$k\to\infty$, by (\ref{limxk}) and since the attracting fixed points of the mappings 
$P_k$ converge to 0, by Proposition \ref{pkconv}. Therefore, 
the series $\sum_{k\geq0}a_kz^k$ converges on the whole complex line $\cc$. The existence is proved. Now let us prove the uniqueness. Let, by contradiction, there exist a series  $\sum a_kz^k$ 
satisfying relations (\ref{rec}), having a positive convergence radius and not coinciding with the one constructed above. Then 
there exists a $k>N$ such that $|\frac{a_{k+1}}{a_k}|>C$, i.e., $x_{k}\notin \overline D_C$. 
For every $l>k$ one has $x_l=P_l^{-1}\circ \dots\circ P_{k+1}^{-1}(x_k)\to\infty$, that is, $\frac{a_{l+1}}{a_l}\to\infty$, 
as $l\to\infty$, by the second statement of Proposition \ref{pkconv}. 
Hence the series diverges everywhere: has zero convergence radius. The contradiction thus obtained proves 
Theorem \ref{cons}.
\end{proof} 

\begin{proof} {\bf of the addendum to Theorem \ref{cons}.} The transformations $P_k$ from Proposition \ref{pkconv} depend 
holomorphically on the parameters. The convergence in the proposition is uniform on compact subsets in the parameter manifold, 
by similar uniform convergence of the sequences $\frac{h_k}{g_k}$ and  $\frac{f_k}{g_k}$ to zero 
(see the condition of the addendum). Then for every 
compact subset $K$ in the parameter manifold there exists a $N>1$ such that for every $k>N$ the mapping $P_k$ are 
contractions of the closed disk $\overline D_C$ for all the parameters from the set $K$, 
with one and the same uniform bound  $\la=\la(K)<1$ for the contraction rate: the proofs of Proposition \ref{pkconv} and 
Theorem \ref{cons} remain valid with uniform convergence in the parameters from the set $K$. The expression under the limit (\ref{convxk}) 
is well-defined and holomorphic in the parameters from the set $K$, and the  convergence in (\ref{convxk}) is uniform on $K$, 
by uniformness of the contraction. This together with Weierstrass Theorem implies that the limit is also holomorphic in 
the parameters from the set $K$. Finally, the sequence $(x_k)$, $x_k\in\cp^1=\oc$ 
depends holomorphically on the parameters, and thus, for every 
$k$ the ratio $x_k=\frac{a_{k+1}}{a_k}$ is a $\oc$-valued holomorphic, hence 
 meromorphic function in the parameters. Fix a $k$ for which it is not identically 
equal to $\infty$ and put $a_k\equiv1$, $a_{k+1}=x_k$. Then the vector $(a_k,a_{k+1})$ depends meromorphically on the 
parameters, and hence, so do all the $a_j$, which are expressed via $(a_k,a_{k+1})$ by linear recurrence relations. 
The poles of the functions $a_j$ are obviously 
contained in the pole divisor of the function $a_{k+1}$. In the case, when the parameter manifold is Stein and contractible, 
every analytic hypersurface (e.g., the pole divisor under question) is the zero locus of a holomorphic function $\Phi$, see 
\cite[chapter VIII, section B, lemma 12]{gr}. 
Replacing all $a_j$ by $\Phi a_j$ yields a series $\sum_ka_kz^k$ depending 
holomorphically on the parameters. The addendum is proved. 
\end{proof} 

%
%

\def\ol{\left(\begin{matrix}  0 & 1\end{matrix}\right)}
\def\lo{\left(\begin{matrix} 1 & 0\end{matrix}\right)}
\def\loc{\left(\begin{matrix}  1 \\  0\end{matrix}\right)}

\section{Explicit formulas for  solutions and the coefficients $d_{j\pm}$}

Here we present explicit formulas for the unique converging series from Theorem \ref{conv} solving recurrence relations (\ref{recur}). 
First in Subsection 3.1 we  provide a general method for writing them, which essentially repeats and slightly generalizes the method 
from \cite[section 3, pp. 337--338]{bt1}. Then we write the above-mentioned formulas 
for $k\to+\infty$, and afterwards for $k\to-\infty$. At the end of the 
section we prove the statement of Theorems \ref{noncom} and \ref{conv} on global holomorphic dependence of the solutions on the 
parameters.

\subsection{Solution of recurrence relation via infinite matrix product:  a general method}

Here we consider a  solution of general recurrence relations (\ref{rec}) 
from Theorem \ref{cons}. Let $g_k$, $f_k$, $h_k$ be the 
coefficients in (\ref{rec}). Let $P_k:\cp^1\to\cp^1$ be the projectivizations of the transformations $\La_k^{-1}$, see 
(\ref{lak}). Let $\sum_ka_kz^k$ be a  solution to (\ref{rec}). Recall that  
$x_k=(a_k:a_{k+1})\in\cp^1\simeq\oc$, in the standard coordinate $w$ on $\oc$ one has $x_k=\frac{a_{k+1}}{a_k}$. We have  
$x_{k-1}=P_k(x_k)$, and for every $k$ the infinite product $P_kP_{k+1}\dots$  converges to $x_{k-1}$. More precisely,  
$P_k\circ\dots\circ P_m(z)\to x_{k-1}$, as $m\to\infty$ uniformly on compact subsets in $\cc=\oc\setminus\{\infty\}$, as in the 
proof of Theorem \ref{cons}. 
One can then deduce  that there exists a number sequence $r_k$ such that for every $k$ the infinite matrix product 
$(r_k\La_k^{-1})(r_{k+1}\La_{k+1}^{-1})\dots$ converges to a rank 1 matrix $\mcr_k$ such that 
\begin{equation}\left(\begin{matrix}  a_{k-1} \\  a_k\end{matrix}\right)=\mathcal R_k\left(\begin{matrix}  1\\  0\end{matrix}
\right).\label{eqm}\end{equation}
The latter relation allows to write an explicit formula for the  solution $\sum_ka_kz^k$ in the following way. 
The infinite product of the matrices $\La_k^{-1}$ themselves diverges, since 
their terms $-\frac{g_k}{f_k}$ tend to infinity: one has to find a priori unknown normalizing constants $r_k$. To construct a converging matrix product explicitly, we will consider a rescaled sequence $a_k$, that is 
 $$c_k=q_ka_k, \ q_{k}\in\cc, \ \frac{q_{k-1}}{q_k}\simeq-\frac{f_k}{g_k}\to0, \text{ as } k\to\infty.$$
 Rewriting relations (\ref{rec}) in terms of the new sequence $c_k$ yields
 \begin{equation} \left(\begin{matrix}  c_{k-1}\\  c_k\end{matrix}\right)=M_k\left(\begin{matrix}  c_k\\  c_{k+1}\end{matrix}\right), 
 \label{bmk}\end{equation} 
 $$M_k=\left(\begin{matrix}  -\frac{q_{k-1}}{q_k}\frac{g_k}{f_k} & -\frac{q_{k-1}}{q_{k+1}} \frac{h_k}{f_k}\\ 1 & 0\end{matrix}\right)=\left(\begin{matrix}  1+o(1) & o(1)\\  1 & 0\end{matrix}\right).$$
 The matrices $M_k$ converge to the projector 
 $$P:\cc^2\to\cc^2, \ P=\left(\begin{matrix}  1 & 0\\  1 & 0\end{matrix}\right).$$ 
 Our goal is to choose the rescaling factors $q_k$ so that the infinite products 
 $$R_k=M_kM_{k+1}\dots$$
 converge: then the limit is a one-dimensional operator $R_k$ with $\ker R_k$ being generated by the vector $(0,1)$. 
It appears that one can achieve the latter convergence by appropriate choice of  normalizing constants $q_k$.

We use the following sufficient conditions of convergence of products of almost projectors $M_k$. 

\begin{lemma} \label{lemkp}  Let $H$ be either a finite dimensional, or a Hilbert space. Let $M_k:H\to H$ be a sequence of bounded operators 
that tend (in the norm) to an orthogonal projector $P:H\to H$. Let 
\begin{equation} M_k=P+S_k, \ \sum_k||S_k||<\infty.\label{condin}
\end{equation}
Then the infinite product $R_k=M_kM_{k+1}\dots$ converges in the norm, and  $\ker P \subset \ker  R_k$. One has $R_k\to P$, 
as $k\to\infty$, in the operator norm, and $\ker R_k=\ker P$ for every $k$ large enough. 
 \end{lemma}
 
 \begin{proof} Fix a $k$ and set $T_n=T_{k,n}=M_k\dots M_n$ for $n\geq k$; $T_k=M_k$.  One has 
 $$T_{n+1}=T_nM_{n+1}=T_n(P+S_{n+1}), \ T_{n+1}-T_nP=T_nS_{n+1}.$$
 The latter equality implies that 
$$ ||T_{n+1}||\leq||T_n||(1+||S_{n+1}||)\leq e^{||S_{n+1}||}||T_n||.$$
This implies that 
\begin{equation}||T_{k,n}||\leq C_k, \ C_k=e^{\sum_{j\geq k}||S_j||}||M_k||.\label{intk}\end{equation}
Now one has 
$$T_{k,n}=T_{k,n-1}P+T_{k,n-1}S_{n},$$
$$T_{k,n+1}=T_{k,n}P+T_{k,n}S_{n+1}=T_{k,n-1}P^2+T_{k,n-1}S_{n}P+T_{k,n}S_{n+1}$$
$$= T_{k,n-1}P+T_{k,n-1}S_{n}P+T_{k,n}S_{n+1}:$$
here we have used the fact that $P$ is a projector, that is,  $P^2=P$. 
The two latter formulas together with (\ref{intk}) 
imply that 
$$||T_{k,n+1}-T_{k,n}||\leq C_k(2||S_{n}||+||S_{n+1}||).$$
The latter right-hand side being a converging series in $n$, the sum of the left-hand sides in $n$ converges and so does 
$T_{k,n}$, as $n\to\infty$, in the operator norm. This also implies that the norm distance of each $T_{k,n}$ 
to the limit $R_k=\lim_{n\to\infty}T_{k,n}$  
 is bounded from above by $\Delta_{k,n}=3C_k\sum_{j\geq n}||S_j||$. Applying this estimate to $T_{k,k}=M_k$, we get 
$dist(M_k,R_k)\leq \Delta_{k,k}=3C_k\sum_{j\geq k}||S_j||$. 
 One has 
$\Delta_{k,n}\to0$, as $n\to\infty$ uniformly in $k$, and also $\Delta_{k,n}\to0$, as $k,n\to\infty$ so that $k\leq n$. This implies that 
$R_k$ and $M_k$ converge to the same limit $P$ in the operator norm, as $k\to\infty$. 
For every $v\in \ker P$ one has $M_nv=S_nv\to0$. Hence, $T_{k,n}v=T_{k,n-1}(S_nv)\to0$, as $n\to\infty$. Therefore, 
$R_kv=0$ and $\ker P\subset\ker R_k$. Let $N>0$ be such that for every $k>N$ one has $||R_k-P||<1$. Let us show 
that $\ker R_k=\ker P$ for these $k$. Indeed, suppose the contrary: $\ker R_k$ is strictly bigger than $\ker P$ for some $k>N$. 
Note that $H=\ker P\oplus P(H)$ (orthogonal decomposition), since $P$ is an orthogonal projector.  Therefore, there exists 
a vector $u_k\in P(H)$ such that $R_k(u_k)=0$. Hence, 
$$||P(u_k)||=||(P-R_k)(u_k)||<||u_k||,$$
while $P(u_k)=u_k$, since $P$ is a projector. The contradiction thus obtained proves the lemma.
\end{proof}

{\bf Addendum to Lemma \ref{lemkp}.} {\it Let in Lemma \ref{lemkp} the operators $S_n$ depend holomorphically on some parameters so 
that the series $\sum_n||S_n||$ converges uniformly on compact subsets in the parameter space. Then the infinite products 
$R_k$ are also holomorphic in the parameters.}

\begin{proof} The  above proof  implies that the sequence $T_{k,n}$ converges uniformly on compact 
subsets in the parameter space. This together with  the Weierstrass Theorem implies the holomorphicity of the limit.
\end{proof}

\begin{corollary} \label{clem} Let 
\begin{equation}
M_k=\left(\begin{matrix}  1+\delta_{11,k} & \delta_{12,k}\\  1 & 0\end{matrix}\right), \ \sum_k|\delta_{ij,k}|<\infty \text{ for  } 
(ij)=(11), (12).
\label{mkdk}\end{equation}
Then the infinite product $R_k=M_kM_{k+1}\dots$ converges, and the  right column of the limit product matrix  $R_k$ vanishes. 
In the case, when $\delta_{12,k}\neq0$ for all $k$, 
the limit matrix $R_k$ has rank 1 for all $k$: its kernel is generated by the vector $(0,1)$. 
In the case, when $\delta_{ij,k}$ depend holomorphically on some parameters and the convergence of the corresponding 
series is uniform on compact sets in the parameter manifold,  the limit $R_k$ is also holomorphic.
 \end{corollary} 
 
 {\bf Addendum  to Corollary \ref{clem}.} {\it In the conditions of the corollary set 
 \begin{equation} c_k=(0, 1) R_k\left(\begin{matrix}  1\\  0\end{matrix}\right).\label{setck}\end{equation}
 Then the sequence $c_k$ is a solution of recurrence relations (\ref{bmk}) such that $\frac{c_k}{c_{k-1}}\to1$, as $k\to\infty$, 
 and one has} 
 \begin{equation} \left(\begin{matrix}  c_{k-1}\\  c_k\end{matrix}\right)=R_k\left(\begin{matrix}  1\\  0\end{matrix}\right).\label{ckk}
 \end{equation}

 \begin{proof} {\bf of Corollary \ref{clem}.} 
 This is the direct application of the lemma and its addendum for the norm induced by appropriate  
 scalar product: the latter product should make the matrix 
 $$P=\left(\begin{matrix}  1 & 0\\  1 & 0\end{matrix}\right)$$
 an orthogonal projector. The kernel $\ker R_k$ contains the kernel $\ker P$, which is generated by the vector $(0,1)$; 
 $\ker R_k=\ker P$, i.e., $rk(R_k)=1$ for all $k$ large enough, by the lemma. In particular, the right column in each $R_k$ vanishes. 
  Now it remains to note that 
 $rk(R_k)=1$  for all $k$, since the matrices $M_k$ are all non-degenerate: $\delta_{12,k}\neq0$.  The corollary is proved.
 \end{proof}

 \begin{proof} {\bf of the Addendum to Corollary \ref{clem}.} Consider the affine chart $\cc=\cp^1\setminus\{(1:-1)\}$ with the  coordinate 
 $w=\frac{z_1-z_2}{z_1+z_2}$ centered at $(1:1)$. 
 The projectivizations $P_k$ of the linear operators $M_k:\cc^2\to\cc^2$ converge to the constant mapping 
 $\cp^1\mapsto(1:1)$ uniformly on compact subsets in $\cc$. Hence, for every $C>0$ there exist a $N=N(C)>0$ and a $\la$, $0<\la<1$ 
 such that for every $k\geq N$  one has 
$P_k(\overline{D_C})\Subset D_C$, and $P_k$ is a $\lambda$-contraction of the disk $\overline{D_C}$, as in the proof of 
Theorem \ref{cons} in the previous section. This together with Theorem \ref{metric} implies that there exists a  sequence 
$(x_k)_{k\geq N(C)}$, $x_k\in\oc=\cp^1$, $w(x_k)\to 0$, as $k\to\infty$, such that $P_k(x_k)=x_{k-1}$ and 
$P_k\circ\dots\circ P_m$ converges to the constant mapping $\cp^1\mapsto x_{k-1}$ uniformly on compact subsets in $\cc$, as $m\to\infty$. Convergence at $(1:0)$ 
implies that $x_{k-1}=(R_{k,11}:R_{k,21})$. Moreover, $R_{k,21}=R_{k+1,11}$, since $R_k=M_kR_{k+1}$ and the matrix $M_k$ has lower 
raw $(1,0)$. The two last statements together imply that the sequence $c_k=R_{k,21}$ satisfies recurrence relations (\ref{bmk}) and 
formulas (\ref{setck}), (\ref{ckk}). 
One has $\frac{c_{k+1}}{c_{k}}\to1$, since $x_k=(c_{k}:c_{k+1})\to(1:1)$. This proves the addendum. 
\end{proof} 

\begin{corollary} \label{coral} Consider recurrence relations (\ref{rec}). Let $q_k\in\cc$ be a sequence such that the  rescaling 
$c_k=q_ka_k$ transforms (\ref{rec}) to (\ref{bmk}). Let the corresponding matrices $M_k$ from (\ref{bmk}) be the same, as in 
(\ref{mkdk}). Let $c_k$ be the same, as in (\ref{setck}). Then the sequence 
$$a_k=q_k^{-1}c_k$$
is a solution of relations (\ref{rec}) for $k\geq1$ 
such that the series $\sum_{k\geq0}a_kz^k$ converges on all of $\cc$. 
\end{corollary}

\begin{proof} The sequence $(a_k)$ is a solution of (\ref{rec}), by construction and the Addendum to Corollary \ref{clem}. 
One has $ \frac{q_{k-1}}{q_k}\simeq-\frac{f_k}{g_k}\to0, \text{ as } k\to\infty$, since the above sequence rescaling 
transforms (\ref{rec}) to (\ref{bmk}). Therefore,  $\frac{a_{k}}{a_{k-1}}\to0$, by the latter statement and since 
$\frac{c_k}{c_{k-1}}\to1$, as was shown above. This implies the convergence of the series $\sum_{k\geq0}a_kz^k$ on $\cc$ 
and proves the corollary. 
\end{proof}

\subsection{Forward  solutions  from Theorems \ref{noncom} and  \ref{conv}}
Here we give explicit formulas for  the solution $\sum_ka_kz^k$ of recurrence relations (\ref{recur}) converging, as $k\to+\infty$. 

{\bf Case 1): $b, b+n\notin\zz$ (i.e., the conditions of Theorem \ref{noncom} hold). } 
Let us invert matrix relation (\ref{mat1}). We get 
\begin{equation}\left(\begin{matrix}  a_{k-1}\\  a_{k}\end{matrix}\right)=W_k \left(\begin{matrix}  a_{k}\\  a_{k+1}
\end{matrix}\right),\label{mat2}\end{equation}
$$W_k=\left(\begin{matrix}  \frac{k+b}{\mu}(1+\frac{\la}{(k+b)(k+b+n-1)})  & 
\frac{k+b+1}{k+b+n-1}\\ 
 1 & 0\end{matrix}\right).$$
To obtain an explicit formula for  solution of relation (\ref{recur}), we will use results of Subsection 3.1. 
To do this, we reduce equation 
(\ref{mat2}) to a similar equation with the matrix in the right-hand side converging to a projector. This is done by   renormalizing the sequence $a_k$ by multiplication by appropriate constants depending on $k\geq0$. Namely, set  
$$c_k=\frac{a_k(b)_{k+1}}{\mu^k}, \ (b)_{l}:=b\dots(b+l-1)=\frac{\Gamma(b+l)}{\Gamma(b)}.$$
Recall that the symbol $(b)_l$ is called the {\it Pochhammer symbol}. Translating relations (\ref{mat2}) in terms of the sequence $c_k$ yields 


$$\left(\begin{matrix} c_{k-1}\\ c_{k}\end{matrix}\right)=M_k
\left(\begin{matrix} c_{k}\\ c_{k+1}
\end{matrix}\right),$$
\begin{equation}M_k=\left(\begin{matrix} 1+\frac{\la}{(k+b)(k+b+n-1)} & 
\frac{\mu^2}{(k+b)(k+b+n-1)}\\ 1 & 0\end{matrix}\right)\label{mat3}\end{equation}
$$=\left(\begin{matrix}  \frac{(b)_{k}}{\mu^{k-1}} & 0\\  0 &\frac{(b)_{k+1}}{\mu^{k}}\end{matrix}\right) W_k
\left(\begin{matrix}  \frac{\mu^k}{(b)_{k+1}} & 0\\  0 &\frac{\mu^{k+1}}{(b)_{k+2}}\end{matrix}\right).$$
The infinite matrix product 
\begin{equation} R_k=M_kM_{k+1}\dots\label{rkmk}\end{equation}
converges and depends analytically on $(\la,\mu, n, b)$ at those points, where the denominators in its definition do not vanish, by 
Corollary \ref{clem}.  
\begin{theorem} \label{th24} Let $b, b+n\notin\zz$. For $k\geq0$ set 
\begin{equation}c_k=\left(\begin{matrix}  0 & 1\end{matrix}\right)R_k\left(\begin{matrix}  1\\  0\end{matrix}\right),\label{ckrk}
\end{equation}
\begin{equation}a_k=\frac{\mu^k}{(b)_{k+1}}c_k\label{ak}\end{equation}
The coefficients $a_k$ satisfy recurrence relations (\ref{recur}) for all $k\geq1$, and the series 
\begin{equation} f_+(z)=\sum_{k\geq1}a_kz^k\label{f+z}\end{equation}
 converges on all of $\cc$.
 \end{theorem}
 
 \begin{proof} The sequence $c_k$ satisfies 
 relations (\ref{mat3}), and $\frac{c_k}{c_{k-1}}\to1$, as $k\to\infty$, by the Addendum  to Corollary \ref{clem}. 
 This implies that $a_k$ satisfy (\ref{recur}).
 The series (\ref{f+z}) converges on all of $\cc$, by Corollary \ref{coral}.  This proves the theorem.
 \end{proof}

{\bf Case 2): some of the numbers $b$ or $b+n$ is an integer.} Set 
\begin{equation}k_{0+}=\max\{ m \in \{ -1-b, 1-b-n\} \ | \ m\in\zz\}.\label{kogeq}\end{equation}
Note that now the product $(b)_l=b(b+1)\dots(b+l-1)$ 
can be equal to zero, and thus, the sequence $a_k$ defined by (\ref{ak}) is not necessarily 
well-defined. Let us modify the above rescaling coefficients relating $a_k$ and $c_k$ as follows. For every $s\leq l+1$ set   
\begin{equation}(b)_{s,l+1}=(b+s)\dots(b+l)=\frac{(b)_{l+1}}{(b)_s}; \ (b)_{0,l+1}=(b)_{l+1}; \ (b)_{s,s}=1.\label{bsl}\end{equation}
Set 
\begin{equation}c_k=a_k\mu^{k_{0+}+1-k}(b)_{k_{0+}+2, k+1} \text{ for every } k> k_{0+}; \ c_{k_{0+}+1}=a_{k_{0+}+1}.\label{cknew}\end{equation}
The  sequence $(a_k)$ satisfies (\ref{mat1}), if and only if the sequence $(c_k)$  satisfies (\ref{mat3}). The above formulas remain 
valid with the same matrices $M_k$, which are well-defined for $k\geq k_{0+}+2$: 
the denominators in its fractions do not vanish. 
Therefore, the infinite product $R_k=M_kM_{k+1}\dots$ is well-defined for the same $k$ in the case under consideration. 
\begin{theorem} \label{th25} Let $b$, $n$, $k_{0+}$ be as above, $M_k$ be as in (\ref{mat3}), $R_k=M_kM_{k+1}\dots$, 
\begin{equation} c_k=R_{k,21} \text{ for } k\geq k_{0+}+2,  \ 
c_{k_{0+}+1}=R_{k_{0+}+2,11},\label{cck}\end{equation}
\begin{equation} a_k=\mu^{k-k_{0+}-1}\frac{c_k}{(b)_{k_{0+}+2,k+1}} \text{ for } k\geq k_{0+}+1, \label{ak01}\end{equation}
\begin{equation}a_{k_{0+}}=\frac{((k_{0+}+b+1)(k_{0+}+b+n)+\la)a_{k_{0+}+1}+\mu(k_{0+}+b+2)a_{k_{0+}+2}}{\mu(k_{0+}+b+n)}.\label{ako2}\end{equation}
The sequence $(a_k)_{k\geq k_{0+}}$ satisfies recurrence relations (\ref{recur}) for $k>k_{0+}$. The series $\sum_{k=k_{0+}}^{+\infty}a_kz^k$ 
converges on all of $\cc^*$. 
\end{theorem}
\begin{proof} The sequence $c_k$ satisfies relations (\ref{mat3}) for $k\geq k_{0+}+2$, by the Addendum to Corollary \ref{clem}. 
Therefore, $a_k$ satisfy relations (\ref{mat1}), 
which are equivalent to (\ref{recur}), see the previous discussion. Formula (\ref{ako2}) is equivalent to 
relation (\ref{recur}) for $k=k_{0+}+1$. The denominator $\mu(k_{0+}+b+n)$ in (\ref{ako2})  does not vanish. In the case, when $b+n\notin\zz$, 
this is obvious. In the case, when $b+n\in\zz$, one has $k_{0+}+b+n\geq1$, by (\ref{kogeq}). 
 The series $\sum_{k\geq k_0} a_kz^k$ converges on $\cc^*$, 
by Corollary \ref{coral}. The theorem is proved.
\end{proof}

\subsection{Backward  solutions}
\def\ha{\hat a}
\def\hc{\hat c}
\def\haa{\hat A}
Here we give explicit formulas for  the solution $\sum_ka_kz^{-k}$ of recurrence relations (\ref{recur}) with $k\to-\infty$.
 Set 
$$m=-k,  \ \ha_m=a_{-m}.$$ 
Relation (\ref{mat1}) in new variables $m$ and $\ha_m$ takes the matrix form 
$$\left(\begin{matrix}  \ha_m\\  \ha_{m-1}\end{matrix}\right)=A_{-m}\left(\begin{matrix}  \ha_{m+1}\\  \ha_{m}
\end{matrix}\right).$$
Writing the latter equation with permuted order of vector components 
(we place $\ha_s$ having smaller indices above) yields the same equation with the new matrix obtained 
from $A_{-m}$ by permutation of lines and columns:
\begin{equation}
\left(\begin{matrix}  \ha_{m-1}\\  \ha_{m}\end{matrix}\right)=
\haa_m \left(\begin{matrix}  \ha_m\\  \ha_{m+1}\end{matrix}\right),\label{hamm}\end{equation}
\begin{equation}\haa_m=\frac{b+n-m-1}{b-m+1}\left(\begin{matrix}  -\frac{\la+(b-m)(b-m+n-1)}{\mu(b-m+n-1)} & 1\\  
\frac{b-m+1}{b-m+n-1} & 0\end{matrix}
\right).\label{mat5}\end{equation}

\medskip
{\bf Case 1): $b, b+n\notin\zz$, as in the conditions of Theorem \ref{noncom}. }
Let us renormalize the sequence $\ha_m$: for $m\geq0$ set 
$$\hc_m=\frac{\ha_m(2-n-b)_{m+1}}{\mu^m}.$$
Translating equation (\ref{hamm}) in terms of the sequence $\hc_m$ yields
\begin{equation} 
\left(\begin{matrix}  \hc_{m-1}\\  \hc_m\end{matrix}\right)=S_m\left(\begin{matrix}  \hc_{m}\\  \hc_{m+1}\end{matrix}\right),
\label{csm}\end{equation}
\begin{equation}S_m=S_m(b,n)=\left(\begin{matrix}  1+\frac{\la-n+2}{(b-m+1)(b-m+n-2)} & \frac{\mu^2(b-m+n-1)}{(b-m+1)(b-m+n-2)(b-m+n-3)}\\
 1 & 0\end{matrix}\right)\label{defsm}\end{equation}
$$=\left(\begin{matrix}  \frac{(2-n-b)_{m}}{\mu^{m-1}} & 0\\  0 & \frac{(2-n-b)_{m+1}}{\mu^{m}}
\end{matrix}\right)\haa_m\left(\begin{matrix}  \frac{\mu^m}{(2-n-b)_{m+1}} & 0\\  0 & 
\frac{\mu^{m+1}}{(2-n-b)_{m+2}}\end{matrix}\right).$$
\begin{theorem} \label{conv-} Let $b,b+n\notin\zz$.  Let the matrices $S_m$ be as above, 
\begin{equation} T_m=S_mS_{m+1}\dots,\label{tmprod}\end{equation}
\begin{equation}\hc_m=\left(\begin{matrix} 0 &1\end{matrix}\right)T_m\left(\begin{matrix} 1\\ 0\end{matrix}\right), \ 
\ha_m=\frac{\hc_m\mu^m}{(2-n-b)_{m+1}}, \ a_k=\ha_{-k}.
\label{cm}\end{equation} 
The sequence $(a_k)$ satisfies recurrence relations (\ref{recur}) for $k\leq-1$, and the series 
\begin{equation} f_-(z)=\sum_{k\leq0}a_kz^{-k}\label{f-z}\end{equation}
converges on $\cc$.
\end{theorem}
\begin{proof}  
The above matrix product converges,  and the sequence $\hc_m$ satisfies equation (\ref{csm}),
by Corollary \ref{clem} and its addendum. This implies that 
  the corresponding sequence $\ha_m$ satisfies (\ref{hamm}), the sequence 
$a_k$ satisfies (\ref{recur}) and the series $f_-(z)$ converges, as in the previous subsection. 
This proves the theorem.
\end{proof}

{\bf Case 2): some of the numbers $b$ or $b+n$ is an integer.} Let 
\begin{equation}k_{0-}=\min\{ r \in \{ -1-b, 1-b-n\} \ | \ r\in\zz\}, \ m_0=-k_{0-}.\label{koleq}\end{equation}
The above Pochhammer symbol may be not well-defined in the case, when $n+b\in\zz$, $2-n-b<0$. We use the 
inequalities
 $$b-m+1\neq0 \text{ for every } m>m_0;$$
 \begin{equation} 2-n-b+m, 3-n-b+m\neq0 \text{ for every } m\geq m_0, 
\label{ineqm}\end{equation}
which follow immediately from (\ref{koleq}). The sequence rescaling 
$$\hc_m=\frac{\ha_m(2-n-b)_{m_0, m+1}}{\mu^m}$$
is well-defined and invertible for all $m\geq m_0$, by (\ref{ineqm}). It 
differs from the previous sequence rescaling from Case 1) 
by multiplication by constant independent on $m$, and hence, transforms (\ref{hamm}) to (\ref{csm}), as 
above. The matrices $S_m$ are well-defined for $m>m_0$: the denominators in their fractions do not vanish, by (\ref{ineqm}).  
Let $T_m$ be their products (\ref{tmprod}) defined for $m>m_0$. 
\begin{theorem} \label{thko-} Let $S_m$ and $T_m$ be the same, as in   (\ref{tmprod}), 
$$\hc_m=\left(\begin{matrix}0 &1\end{matrix}\right)T_m\left(\begin{matrix} 1\\  0\end{matrix}\right) \text{ for } 
 m>m_0, \ c_{m_0}=T_{m_0+1,\ 11},$$
 \begin{equation}
\ha_m=\frac{\hc_m\mu^m}{(2-n-b)_{m_0,m+1}} \text{ for } m\geq m_0, \ a_k=\ha_{-k}.
\label{cm2}\end{equation} 
The sequence $(a_k)_{k\leq k_{0-}}$  satisfies recurrence relations (\ref{recur}) for $k<k_{0-}$, 
and the series $\sum_{k\leq k_{0-}}a_kz^{-k}$ 
converges on $\cc^*$.
\end{theorem}
 The proof of Theorem \ref{thko-} repeats the proof of Theorem \ref{conv-} with obvious changes. 

\subsection{Theorem \ref{noncom}: formulas for $d_{0\pm}$ and $d_{1\pm}$} 

\begin{lemma} Let $(n,\la,\mu,b)\in U$, see (\ref{defu}). 
Let $f_+(z)=\sum_{k\geq1}a_kz^k$ and 
$f_-(z)=\sum_{k\leq0}a_kz^{-k}$ be the functions from (\ref{f+z}) and (\ref{f-z}) 
constructed in the two previous subsections, case 1). Then 
\begin{equation}z^{-b}\mcl(z^bf_{\pm}(z^{\pm1}))=d_{0\pm}+d_{1\pm}z,\label{diffop}\end{equation}
\begin{equation} d_{0+}=\mu(b+1)a_1; \ d_{1+}=((b+1)(b+n)+\la)a_1+\mu(b+2)a_2,\label{d+f}\end{equation}
where $a_1$ and $a_2$ are the same, as in (\ref{ak});
\begin{equation} d_{0-}=(b(b+n-1)+\la)a_0-\mu(b+n-1)a_{-1}; \ d_{1-}=-\mu(b+n)a_0,\label{d-f}\end{equation}
where $a_{-1}$ and $a_0$ are the same, as in (\ref{cm}).
\end{lemma}

\begin{proof} The left-hand side in (\ref{diffop}) with index ``$+$'' is a Taylor series with coefficients at $z^k$ being 
equal to the left-hand side of the corresponding recurrence relation (\ref{recur}). The latter relation holds for all $k\geq2$, by construction. 
This implies (\ref{diffop}) with $d_{0+}$, $d_{1+}$ being equal to the left-hand sides of relations (\ref{recur}) for $k=0$ and $k=1$ 
respectively. This implies (\ref{d+f}). The proof for the index ``$-$'' is analogous.
\end{proof}

\subsection{End of proof of Theorems \ref{noncom} and \ref{conv}: holomorphic dependence of solutions on the parameters}

\begin{proposition} The solutions $(a_k)$ of recurrence relations (\ref{recur}) 
constructed above via infinite matrix products depend holomorphically 
on the parameters from the domain, where all the factors  of the matrix product are well-defined. 
\end{proposition}

The proposition follows immediately from construction and Corollary \ref{clem}. It implies the statements of Theorems \ref{noncom} 
and \ref{conv} on holomorphic dependence of the corresponding solutions on the parameters.

\section{Application:  monodromy eigenvalues}

Here we study the eigenfunctions of the monodromy operator of  Heun equation (\ref{heun}). Recall that the monodromy operator 
of a linear differential equation on the Riemann sphere acts on the space of germs of  its solutions at a nonsingular point 
$z_0$. Namely, fix a closed path $\alpha$ starting at $z_0$ in the complement to the singular points of the equation. 
The monodromy operator along the path $\alpha$ sends each germ to the result of its analytic extension along the path $\alpha$. 
It is completely determined by the homotopy class of the path $\alpha$ in the complement to the singular points of the equation. 
In the case under consideration of double confluent Heun equations (\ref{heun}) there are exactly two singular points: zero and 
infinity. By the {\it monodromy operator of double confluent Heun equation} (\ref{heun}) we mean the monodromy operator along a 
counterclockwise circuit around zero. Each monodromy eigenfunction with an eigenvalue $e^{2\pi i b}$ 
has the form of a series 
\begin{equation} E(z)=\sum_{k\in\zz}a_kz^{k+b}, \ b\in\cc,\label{solb}\end{equation}
 converging on $\cc^*$. Here we write down an explicit analytic equation on 
those $b$, for which the latter solution $E(z)$ of equation (\ref{heun}) exists, i.e., 
there exists a bi-infinite sequence $(a_k)_{k\in\zz}$ satisfying recurrence relations (\ref{recur}) such that the  
the bi-infinite series $\sum_{k\in\zz}a_kz^k$ converges on $\cc^*$. 

We consider different cases, but the method of finding the above $b$ is general for all of them. The coefficients $a_k$ 
with $k\to+\infty$ should form a unique converging series 
(up to constant factor) that satisfies recurrence relations (\ref{recur}). 
Similarly, its coefficients with $k\to-\infty$ should form a unique converging series satisfying (\ref{recur}). 
Finally, the above positive and negative parts of the series should paste 
together and form a solution of Heun equation. In the simplest, non-resonant case, when $b,b+n\notin\zz$, the pasting equation is given by  (\ref{schivka}). 
The coefficients $a_k$, $k\geq1$ satisfying (\ref{recur}) for $k\geq2$ and forming a converging series are given by formulas (\ref{ak}); 
the sequence $(a_k)_{k\leq0}$, $a_k=\ha_{-k}$, satisfying (\ref{recur}) for $k<0$ and forming a converging series is given by formula (\ref{cm}). 

It appears that substituting the above-mentioned formulas for $a_k$ to formulas  (\ref{d+f}) and (\ref{d-f}) 
for $d_{j\pm}$ and then substituting the latter formulas  to (\ref{schivka}) 
yields a rather complicated pasting equation. To obtain a simpler formula, we proceed as follows. In the non-resonant case we 
extend the sequence $(a_k)_{k\geq1}$ to $k=0$ by putting appropriate $\alpha\in\cc$ instead of $a_0$ (we get $\alpha, a_1,a_2,\dots$)
so that the longer sequence thus obtained satisfies (\ref{recur}) also for $k=1$. Similarly, we extend the sequence $(a_k)_{k\leq0}$ 
to $k=1$ by putting appropriate $\beta\in\cc$ instead of $a_1$ (we get $\dots a_{-1}, a_0,\beta$)
 in order to satisfy equation (\ref{recur}) for $k=0$. The positive and negative series 
 thus constructed paste together to a converging bi-infinite series $\sum_{k\in\zz}a_kz^k$ satisfying (\ref{recur}) (after their rescaling by 
 constant factors), if and only if 
 \begin{equation}\alpha\beta=a_0a_1.\label{aoal}\end{equation}
We obtain an explicit expression for equation (\ref{aoal}). 
%

In what follows, we use the two next propositions. 

\begin{proposition} \label{monprod}  The determinant of the monodromy operator of Heun equation (\ref{heun}) equals $e^{-2\pi i n}$.  
\end{proposition}

\begin{proof} We prove the statement of the proposition for $\mu\neq0$: then it will follow automatically for $\mu=0$, by continuity. 
  The monodromy matrix is the product of the formal monodromy matrix $\diag(1, e^{-2\pi i n})$ and a pair of unipotent 
matrices: the inverse to the Stokes matrices, cf. \cite[formulas (2.15) and (3.2)]{4}. Therefore, its determinant equals $e^{-2\pi i n}$. 
Another possible proof would be to use the formula for Wronskian of two linearly independent solutions of equation (\ref{heun}) 
from \cite[p. 339, proof of theorem 4]{bt1}. It shows that the Wronskian equals $z^{-n}$ times a function holomorphic on $\cc^*$, and 
hence, it gets multiplied by $e^{-2\pi in}$ after analytic continuation along a positive circuit around zero. 
\end{proof}

Recall, see \cite[equations (32), (34)]{tert2}, \cite[p. 336, lemma 1]{bt1} that the transformation $\#:E\mapsto\# E$: 
\begin{equation}(\#E)(z):=2\omega z^{-n}(E'(z^{-1})-\mu E(z^{-1})), \ \la+\mu^2=\frac1{4\omega^2}, \ \omega>0\label{defdi}\end{equation}
is an involution acting on the space of solutions of equation (\ref{heun}). 

\begin{proposition} \label{propdiez} Let the monodromy operator of Heun equation have distinct eigenvalues. Then the involution $\#$ permutes the 
corresponding eigenfunctions.
\end{proposition}

\begin{proof} The involution under question is a composition of transformation of a function to 
its linear combination with its derivative,  the variable change $z\mapsto z^{-1}$ and  multiplication by $z^{-n}$. Let now $E$ be a monodromy eigenfunction with eigenvalue 
$\la_1$. The composition of the first and second operations transforms $E$ to a function, whose analytic extension along 
a counterclockwise closed path  around zero multiplies it by $\la_1^{-1}$: the second operation inverses the direction of the path. 
The multiplication by $z^{-n}$ multiplies the above result of analytic extension by $e^{-2\pi i n}$. Therefore, $\# E$ is a monodromy 
eigenfunction with the eigenvalue $\la_2=\la_1^{-1}e^{-2\pi i n}$.  It coincides with the second monodromy eigenvalue, 
since it is found by the condition that $\la_1\la_2=e^{-2\pi i n}$, see Proposition \ref{monprod}. This proves the proposition.
\end{proof}


\subsection{Nonresonant case:  $b, b+n\notin\zz$}
In this case the denominators in formulas (\ref{mat3}) and (\ref{defsm}) for the matrices $M_k$ and $S_m$ respectively 
are nonzero for all integer $k$ and $m$, and hence, the matrices are well-defined together with the infinite products 
$R_k=M_kM_{k+1}\dots$, $T_m=S_mS_{m+1}\dots$. 

\begin{theorem} \label{eight}  Let  $b, b+n\notin\zz$. Equation (\ref{heun}) has a monodromy eigenfunction with eigenvalue 
$e^{2\pi i b}$, $b\in\cc$, if and only if 
\begin{equation}(b+1)(b+n-2)R_{1,11}T_{0,11}+\mu^2R_{1,21}T_{0,21}=0.\label{paste}\end{equation}
\end{theorem}

\begin{proof} Let $f_+(z)=\sum_{k\geq1}a_kz^k$ be a converging series satisfying (\ref{recur}) for $k\geq2$.  
Recall, see (\ref{ak}), that 
$$a_1=\frac{\mu}{b(b+1)}\ol R_1\loc.$$
Let us extend formula (\ref{ak}) to $k=0$: set 
$$\alpha=\frac1b\ol R_0\loc=\frac1b\lo R_1\loc.$$
The sequence $\alpha, a_1,a_2,\dots$ satisfies (\ref{recur}) for $k\geq1$, by Theorem \ref{th24}. Recall, see (\ref{cm}), that 
$$\hc_0=\left(\begin{matrix} 0 &1\end{matrix}\right)T_0\left(\begin{matrix} 1\\  0\end{matrix}\right), \ 
\ha_0=a_{0}=\frac{\hc_0}{2-n-b}.$$
Let us extend formula (\ref{cm}) to $k=1$, $m=-1$: set 
$$\beta=\mu^{-1}\ol T_{-1}\loc= \mu^{-1}\left(\begin{matrix} 1 &0\end{matrix}\right)T_0\left(\begin{matrix} 1\\  0\end{matrix}\right).$$
The sequence $\dots,a_{-2},a_{-1},a_0,\beta$ satisfies (\ref{recur}) for all $k\leq0$, by Theorem \ref{conv-}. Substituting the above  
formulas for $a_1$, $\alpha$, $a_0$, $\beta$ to  pasting equation (\ref{aoal}) yields (\ref{paste}). The theorem is proved.
\end{proof} 

%

\subsection{Resonant case: at least one of the numbers $b$, $b+n$ is integer}
\begin{proposition} \label{p1res} Let at least one of the numbers $b$, $b+n$ be integer. 
Then  Heun equation (\ref{heun}) has a solution of type 
$E(z)=z^bf(z)$ with  $f(z)$ being a holomorphic function on $\cc^*$, if and only 
if it has a solution holomorphic on $\cc^*$, i.e., corresponding to $b=0$.  In this case the monodromy eigenvalues are 1 and 
$e^{-2\pi i n}$. 
\end{proposition}


\begin{proof} Let the above solution $E$ exist. Then it is a monodromy eigenfunction with the eigenvalue $e^{2\pi i b}$. 
The other eigenvalue equals $e^{-2\pi i(b+n)}$, by Proposition  \ref{monprod}. At least one   eigenvalue equals one, since 
either $b$, or $b+n$ is integer, by 
assumption. The monodromy eigenfunction corresponding to unit eigenvalue is holomorphic on $\cc^*$. Conversely, a solution 
holomorphic on $\cc^*$ is a solution $E$ as above with $b=0$. It is a monodromy eigenfunction with unit eigenvalue. Then the other 
eigenvalue equals $e^{-2\pi i n}$, by Proposition \ref{monprod}. This proves proposition \ref{p1res}.
\end{proof}

\begin{corollary} 
A solution $E$ as in Proposition \ref{p1res} exists, if and only if the recurrence relations  (\ref{recur}) 
with $b=0$: 
\begin{equation} (k(k+l)+\la)a_k-\mu(k+l)a_{k-1}+\mu(k+1)a_{k+1}=0, \ l=n-1\label{recur0}\end{equation}
 have a solution $(a_k)_{k\in\zz}$ such that the series $\sum_{k\in\zz}a_kz^k$ converges on $\cc^*$. 
\end{corollary}

\begin{proposition} \label{propeq}  Every semiinfinite sequence  
$(a_k)_{k\geq-2}$ satisfying equations (\ref{recur0}) for $k\geq-1$  (without convergence condition) satisfies the relation
\begin{equation}(1-l+\la)a_{-1}-\mu(l-1)a_{-2}=0.\label{a12}\end{equation}
\end{proposition}

\begin{proof} Equation (\ref{a12}) coincides with (\ref{recur0}) for $k=-1$. 
\end{proof}

\begin{proposition} \label{corel} Let $l\notin\zz$, $n=l+1$. 
A solution of Heun equation (\ref{heun}) holomorphic 
on $\cc^*$ exists, if and only if at least one of the two following statements holds:

- either the unique  semiinfinite sequence $(a_k)_{k\leq -1}$  solving relations (\ref{recur0}) for $k\leq-2$  with  series  $\sum_{k=-\infty}^{-1}a_kz^k$ converging on $\cc^*$  satisfies relation (\ref{a12}); 

- or Heun equation (\ref{heun}) has an entire solution: holomorphic on $\cc$. 
\end{proposition}

\begin{proof} Let 
\begin{equation} f(z)=\sum_{k=-\infty}^{-1}a_kz^k\label{k-}\end{equation}
be a semiinfinite solution of recurrence relations (\ref{recur0}) for $k\leq-2$. 
Note that for every $k\leq-1$ its coefficient $a_k$ is 
uniquely determined as a linear combination of the two previous ones $a_{k-2}$ and $a_{k-1}$, see (\ref{recur0}) 
written for $k\leq-2$, since $\mu\neq0$, see (\ref{heun}).  The same holds in the opposite direction: for every $k\leq-3$ the coefficient 
$a_k$ is expressed as a linear combination of the coefficients 
$a_{k+1}$ and $a_{k+2}$ by (\ref{recur0}), since $l\notin\zz$. The two latter  statements together imply that  $a_{-2}$ and 
$a_{-1}$ do not both vanish. Therefore, the above negative semiinfinite series can be extended to positive $k$ as a (may be just formal) 
two-sided solution of (\ref{recur0})  only in the case, when relation (\ref{a12}) holds. Let us show that 
in this case it does extend to a true (not just formal) two-sided solution. 

Note that $a_{-1}\neq0$, by relation (\ref{a12}) and since $a_{-2}$, $a_{-1}$ do not vanish both, $\mu\neq0$   and $l\neq1$. 
Equation (\ref{recur0}) with $k=-1$ has zero multiplier at $a_0$, see (\ref{a12}),  and hence, 
holds for arbitrary $a_0$. The same equation with $k=0$ yields 
\begin{equation}\la a_0-l\mu a_{-1}+\mu a_1=0.\label{a011}\end{equation}
This is a linear non-homogeneous equation on the pair $(a_0,a_1)$. Hence, its solutions form a line $L_1\subset\cc^2$ 
that does not pass through 
the origin: $a_{-1}\neq0$. On the other hand, the
 pairs $(a_0,a_1)$ extendable to true (not just formal) semiinfinite  solutions in positive $k$ exist and 
are all proportional 
(uniqueness of solution up to constant factor and since for every $k\geq1$ the coefficient 
$a_{k\pm1}$ is uniquely determined by $a_k$ and 
$a_{k\mp1}$ via relations (\ref{recur0}), since $l\notin\zz$). Hence, they form a line $L_0$ through 
the origin. Let us choose $(a_0,a_1)$ to be the intersection of the above lines $L_0$ and $L_1$, provided they are not parallel 
(the case of parallel lines is discussed below). Then the pair $(a_0,a_1)$ extends to a semiinfinite  solution of relations 
(\ref{recur0}) in positive $k$, by construction. The complete Laurent series $\sum_{k=-\infty}^{+\infty}a_kz^k$ thus constructed is 
a  solution to equations (\ref{recur0}) and hence, to Heun equation (\ref{heun}).

{\bf Case, when $L_0$ and $L_1$ are parallel.} In this case $L_0=\{\la a_0+\mu a_1=0\}$, and $(a_0,a_1)$ defines a  solution to (\ref{recur0}) with positive 
$k$, if and only if $(a_0,a_1)\in L_0$.  This solution extends to negative $k$ by putting $a_k=0$ for $k<0$, since relation
 (\ref{recur0}) for $k=-1,0$ is equivalent to  (\ref{a12}) and (\ref{a011}) respectively. Finally we obtain a 
{\it converging  Taylor series} satisfying (\ref{recur0}) and hence, presenting a solution of Heun equation (\ref{heun}) 
holomorphic on $\cc$.  Proposition \ref{corel} is proved.
\end{proof}

The next theorem describes those parameter values for which Heun equation (\ref{heun}) has an entire solution. To state it, 
consider the following matrices $M_k$, $R_k$ and numbers $a_k$, $\xi_l$: 
  $$M_k=\left(\begin{matrix}  1+\frac{\la}{k(k+l)} & \frac{\mu^2}{k(k+l)}\\  1 & 0\end{matrix}\right),  \ 
R_k=M_kM_{k+1}\dots \text{ for } k\geq1,$$
$$a_k=\frac{\mu^k}{k!}R_{k,21} \text{ for } k\geq1, \ a_0=R_{1,11},$$
\begin{equation}\xi_l=\xi_l(\la,\mu)=\la a_0+\mu a_1=\la R_{1,11}+\mu^2R_{1,21}.\label{xil}\end{equation}

\begin{theorem} \label{thol} A Heun equation (\ref{heun}) with $n=l+1$, $l\in\cc\setminus\zz_{<0}$  
has an entire solution, if and only if $\xi_l(\la,\mu)=0$.
\end{theorem}
 Theorem \ref{thol} is equivalent to Corollary \ref{cxi}. It was partly proved and conjectured in \cite[section 3, theorem 2]{bt1} and proved completely for entire $l\geq0$ in 
 \cite[subsection 3.1, theorem 3.5]{bg}. 
For completeness of presentation  let us give its direct proof without using results of loc. cit. 
 
  \begin{proof} {\bf of Theorem \ref{thol}.}  The above matrices $M_k$ and numbers $a_k$ coincide with those from (\ref{mat3}) 
  and (\ref{ak01}) respectively constructed for recurrence relations (\ref{recur}) with $b=0$, $n=l+1$, 
  \begin{equation} (k(k+l)+\la)a_k-\mu(k+l)a_{k-1}+\mu(k+1)a_{k+1}=0;\label{b=0}\end{equation}
  here $k_{0+}=-1$. The matrices $M_k$ are well-defined for all $k\in\nn$, whenever $l\notin\zz_{<0}$. 
  (If $l=0$, then $k_{0+}=0$, but the corresponding sequence $a_k$ from (\ref{ak01}) remains the same, as in (\ref{xil}),  
  up to constant factor.) This together with Theorem \ref{th25} implies that 
  the sequence $(a_k)$ satisfies (\ref{b=0}) for $k\geq1$ and the series $E(z)=\sum_{k=0}^{+\infty}a_kz^k$ 
  converges on $\cc$.  Therefore, 
  $\mcl E=const$, and the latter constant is the left-hand side of the relation (\ref{b=0}) corresponding to $k=0$: that is, 
  $\la a_0+\mu a_1=\xi_l(\la,\mu)$. This together with the uniqueness of an entire function $E$ for which $\mcl E=const$ 
  (Theorem \ref{xi=0}) implies the statement of Theorem \ref{thol}.
  \end{proof}


%

\begin{theorem} Let $n\notin\zz$,  $\mu\neq0$,   
$$S_m=\left(\begin{matrix}  1+\frac{\la-n+2}{(1-m)(n-m-2)} & \frac{\mu^2(n-m-1)}{(1-m)(n-m-2)(n-m-3)}\\
 1 & 0\end{matrix}\right) \text{ for } m\geq2,$$
$$T_m=S_mS_{m+1}\dots.$$
Heun equation (\ref{heun}) has a solution holomorphic on $\cc^*$, if and only if either $\xi_l(\la,\mu)=0$, or 
\begin{equation} (2-n+\la)(4-n)T_{2,11}-\mu^2(n-2)T_{2,21}=0.\label{pastreson}\end{equation}
\end{theorem}

\begin{proof} A solution of Heun equation holomorphic on $\cc^*$ exists if and only if some of the two statements of 
Proposition \ref{corel} holds. The second one, the existence of an entire solution is equivalent to the equation $\xi_l(\la,\mu)=0$, 
by Theorem \ref{thol}. Let us show that the first statement of Proposition \ref{corel}, that is, 
equation (\ref{a12}) on the coefficients $a_{-2}$, 
$a_{-1}$ of the backward solution of recurrence relations (\ref{recur0}) is equivalent to (\ref{pastreson}). To do this, let us recall the 
formulas for the sequence $a_k$ (up to common constant factor), see (\ref{cm2}): 
$$\hc_m=T_{m,21}=T_{m+1,11}, \ \hc_1=T_{2,11},$$
$$\ha_m=\frac{\hc_m\mu^m}{(2-n)_{m+1}}=a_{-m}, \ m\geq1.$$
The sequence $(a_k)_{k\leq-1}$ satisfies recurrence relations (\ref{recur0}) for $k\leq-2$, as in Subsection 3.3, and the series 
$\sum_{k\leq-1}a_kz^{-k}$ converges on $\cc$: here we have rewritten the formulas from Subsection 3.3 for $b=0$. One has 
$$a_{-2}=\frac{\mu^2T_{2,21}}{(2-n)(3-n)(4-n)}, \ a_{-1}=\frac{\mu T_{2,11}}{(2-n)(3-n)},$$
by definition. Substituting the latter formulas and $l=n-1$ to (\ref{a12}) yields (\ref{pastreson}). This together with Proposition 
\ref{corel} proves the theorem.
\end{proof}  

%
%


\def\wt#1{\widetilde#1}
\subsection{Double resonant subcase: $n,b\in\zz$} 

Recall that we study the existence of solution (\ref{solb}) of Heun  equation of the type 
(\ref{heun}). In the case under consideration $b\in\zz$, 
and without loss of generality we can and will consider that $b=0$. In this case a solution we are looking for is holomorphic 
on $\cc^*$ and presented by a Laurent series $E(z)=\sum_{k\in\zz}a_kz^k$ converging on $\cc^*$. Recall that $l=n-1$. 
Without loss of generality we will consider that $l\geq0$. One can achieve this by applying the transformation 
$$\Diamond:E(z)\mapsto e^{\mu(z+z^{-1})}E(-z^{-1}),$$
which is an isomorphism of the solution space of equation (\ref{heun}) and equation 
\begin{equation} \mcl E=z^2E''+((-l+1)z+\mu(1-z^2))E'+(\la+\mu(l-1)z)E=0,\label{heun2}\end{equation}
see \cite[section 4, formula (39)]{bt1}.  It sends solutions of equation (\ref{heun}) 
holomorphic on $\cc^*$ onto solutions of equation (\ref{heun2}) holomorphic on $\cc^*$. 

\begin{theorem} \label{talt} Let $l\in\zz$, $l\geq0$, $\mu\neq0$. Equation (\ref{heun}) with $n=l+1$ has a solution holomorphic on $\cc^*$, if and only if 
its monodromy is unipotent. This happens, if and only if equation (\ref{heun}) satisfies one of the two following statements:

1) either it has an entire solution, i.e., holomorphic on $\cc$; this holds if and only if the monodromy is trivial; 

2) or the corresponding equation (\ref{heun2}) has a nontrivial polynomial solution.

In the case, when the parameters $\la$ and $\mu$ are real and $\mu>0$, statements 1) and 2) are incompatible: if statement 2) 
holds, then the monodromy is nontrivial (a unipotent Jordan cell).
\end{theorem}

\begin{remark} The incompatibility of statements 1) and 2)  for real parameter values 
was proved in \cite[theorem 3.10]{bg}. It follows from our result on positivity of determinants of modified Bessel functions \cite[theorem 1.3]{bg} and results of \cite{bt2}. (Incompatibility for real parameters was proved for $\mu>0$, but it holds whenever $\mu\neq0$: 
cases $\mu>0$ and $\mu<0$ are symmetric and one is reduced  to the other via appropriate transformation of the equation.) 
However as it will be shown below in Proposition \ref{pbes}, statements 1) and 2) hold simultaneously for an infinite set of {\it complex} parameter values  already for $l=1$. 
\end{remark}

Theorem \ref{talt} will be proved below. The sets of parameter values for which statements 1) or 2) hold were already described in 
\cite{bt0, bt1, bg}. The set of parameters corresponding to Heun equations (\ref{heun}) having entire solutions is given by 
equation $\xi_l(\la,\mu)=0$, see Theorem \ref{thol}. 
Let us recall the description of the parameters corresponding to equations (\ref{heun2}) with polynomial solutions.
   To do this, consider the three-diagonal 
  $l\times l$-matrix 
 \begin{equation}H=\left(\begin{matrix}  0 & \mu & 0 & 0  & 0 & 0 \dots & 0\\
  \mu(l-1) & 1-l & 2\mu & 0 & 0 & \dots & 0\\
  0 & \mu(l-2) & -2(l-2) & 3\mu & 0 & \dots & 0\\
  \dots & \dots & \dots & \dots & \dots & \dots & \dots\\
  0 &\dots & 0 & 0 & 2\mu & -2(l-2) & (l-1)\mu\\
  0 & \dots & 0 & 0 & 0 & \mu & 1-l \end{matrix}\right):  \label{defh}\end{equation}
$$H_{ij}=0 \text{ if } |i-j|\geq2; \ H_{jj}=(1-j)(l-j+1); \ H_{j,j+1}=\mu j; \ H_{j,j-1}=\mu(l-j+1).$$
 The matrix $H$ belongs to the class of the so-called Abelian matrices that arise in 
 different questions of mathematics and mathematical physics\footnote{Ilyin, V.P.; Kuznetsov, Yu.I. {\it Three-diagonal matrices 
 and their applications.} Moscow, Nauka, 1985.}. 

 \begin{theorem} \label{tpol} \cite[section 3]{bt0} A Heun equation (\ref{heun2})  with $\mu\neq0$  has a polynomial solution, if 
 and only if $l\in\nn$ and the three-diagonal matrix $H+\la Id$ has zero determinant.
 \end{theorem}

\begin{remark} Let us  explain why 
a Heun equation (\ref{heun2}) cannot have a polynomial solution for $l\notin\nn$, $\mu\neq0$. Indeed, 
 the corresponding three-term relations are of the form (\ref{b=0}) with $l$ replaced by $-l$. For every $k\geq1$ the coefficients at 
 $a_{k\pm1}$ in these relations are non-zero.   Hence their solution cannot be a polynomial, and it is an infinite series. 
 \end{remark}

\begin{proof} {\bf of Theorem \ref{talt}.} The monodromy of equation (\ref{heun}) has unit determinant, since $l\in\zz$ and by 
Proposition \ref{monprod}. Let equation (\ref{heun}) have 
a solution $E(z)$ holomorphic on $\cc^*$.  The latter solution  is a monodromy eigenfunction with eigenvalue one. 
Hence, the other monodromy eigenvalue also equals one, by unimodularity of the monodromy,  Conversely, let the monodromy be 
unipotent. Then it has an eigenfunction with  eigenvalue one, and hence, the  latter eigenfunction is a solution of 
equation (\ref{heun}) holomorphic on $\cc^*$. 

Now let us show that existence of a solution $E(z)$ holomorphic on $\cc^*$ is equivalent to the statement that 
one of the two 
statements 1) or 2) holds. Let the above solution $E=\sum_{k\in\zz}a_kz^k$ exist.   
Let us prove that one of 
the two statements 1) or 2) holds. 

The Laurent series  of the solution $E(z)$ converges on $\cc^*$ and the coefficients $a_k$ satisfy recurrence relations (\ref{b=0}). 
%
For $k=-l$ and $k=-1$ respectively these relations take the form 
\begin{equation}\la a_{-l}+\mu(1-l)a_{1-l}=0,\label{-l}\end{equation}
\begin{equation} (1-l+\la)a_{-1}-\mu(l-1)a_{-2}=0.\label{-1}\end{equation}
In particular, they do not contain $a_j$, $j\notin\{-l,\dots,-1\}.$ 
Therefore, given a solution holomorphic on $\cc^*$ of Heun equation (\ref{heun}), its Laurent 
coefficients $a_k$ with $-l\leq k\leq -1$ should form a vector $(a_{-l},\dots,a_{-1})$ satisfying equations (\ref{-l}), (\ref{-1}) and 
the $l-2$ recurrence equations (\ref{b=0}) for intermediate $k=-l+1,\dots,-2$. In other terms,  the latter vector should be in the kernel of 
the three-diagonal $l\times l$- matrix $\wt H$ of equations (\ref{b=0}) with $k=-1,\dots,-l$: its line number  $-k$ 
consists of the coefficients of the $k$-th relation; the coefficient at $a_{-j}$ stands at the column number $j$. 

\begin{proposition} \label{trconj}  Let  $\wt H$ be the latter matrix, and let $H^t$ be the transposed matrix (\ref{defh}). One has 
$$ \wt H=Q(H^t+\la Id)Q^{-1}, \ Q=\left(\begin{matrix}  0 & 0 & 0 & \dots & 1\\
 0 & \dots & 0 & -1 & 0\\
 \dots & 0 & 1 & 0 & 0\\
\dots & \dots &\dots &\dots &\dots\\
  (-1)^{l-1} & 0 & 0 & \dots & 0\end{matrix}\right).$$
 \end{proposition}
 
The statement of Proposition \ref{trconj} can be checked immediately. 

{\bf Case 1). There exists a solution $E(z)$ 
of Heun equation (\ref{heun}) 
 holomorphic on $\cc^*$ 
such that either $l=0$ in (\ref{heun}), or $l\geq1$ and $a_j=0$ for all $j\in\{-l,\dots,-1\}$.} 
Then the series $\hat E(z)=\sum_{k\geq0}a_kz^k$ is an entire solution, i.e.,  holomorphic on $\cc$: it satisfies relations (\ref{b=0}) for all 
$k\in\zz$. Indeed,  relations (\ref{b=0}) hold for all $k>0$, by assumption. For $k<0$ they do not contain $a_j$ with $j\geq0$, by 
(\ref{-1}), and hence, hold automatically, if we put $a_j=0$ for all $j<0$. Relation (\ref{b=0}) for $k=0$ takes the form $\la a_0+\mu a_1=0$, 
and hence, does not contain $a_{-1}$ and holds automatically for the series $\hat E$, as for $E$: for $l=0$ this is obvious; for 
$l\geq1$ this follows from the assumption that $a_{-1}=0$ in $E$. 
A priori, it may happen that the series $\sum_{k=0}^{+\infty}a_kz^k$ is identically zero: in our assumptions, this holds exactly when 
$a_j=0$ for every $j\geq-l$. In this case 
the function $\# E(z)=2\omega z^{-(l+1)}E(z^{-1})$, which is also a solution of equation (\ref{heun}), is an entire solution linearly independent with $E$. It is known that if (\ref{heun}) has an entire solution and $l\in\zz_{\geq0}$, 
then each solution of equation (\ref{heun}) is holomorphic on $\cc^*$,  
and if $l\geq1$, then its Laurent series does not contain monomials $z^j$, $j\in\{-l,\dots,-1\}$, see \cite[lemma 3, statement 6]{bt2}. 
Hence, the monodromy is trivial. 

Let us prove the converse: if the monodromy of equation (\ref{heun}) 
is trivial, then $l\in\zz$ and equation (\ref{heun}) has an entire solution. Indeed, 
 Heun equation (\ref{heun}) is analytically equivalent to the  system of equations 
\begin{equation}\begin{cases} & v'=\frac1{2i\omega z}u\\
& u'=z^{-2}(-(lz+\mu(1+z^2))u+\frac z{2i\omega}v)\end{cases},\label{ttyy}\end{equation}
where $E(z)=e^{\mu z}v(z)$. 
The formal normal form at the origin of system (\ref{ttyy}) is the system 
$$\begin{cases} & \hat v'=0\\
 & \hat u'=-z^{-2}(lz+\mu)\hat u\end{cases}.$$
The monodromy  matrix of system (\ref{ttyy}) written in a canonical sectorial solution base in appropriate sector is the product of three matrices: the 
monodromy $\diag(1,e^{-2\pi i l})$ of the 
formal normal form and two unipotent matrices, one upper-triangular and the other lower-triangular (the inverse to the Stokes matrices). 
See \cite[formulas (2.15) and (3.2)]{4} for more detail.  The latter product is identity, if and only if $l\in\zz$ and the Stokes matrices 
are trivial, as in loc. cit. Triviality of the Stokes matrices is equivalent to the existence of a variable change 
$(v,u)=H(z)(\hat v,\hat u)$ 
transforming system (\ref{ttyy}) to its formal normal form, where $H:\cc\to GL_2(\cc)$ is a holomorphic mapping, 
$H(0)=Id$, as in loc. cit. The formal normal form has an obvious holomorphic 
solution $(\hat v(z),\hat u(z))\equiv(1,0)$. Its image under the latter variable change is a solution $(v(z),u(z))$  of 
system (\ref{ttyy}) holomorphic on $\cc$.  
Therefore, the corresponding solution $E(z)=e^{\mu z}v(z)$ of Heun equation (\ref{heun}) is also holomorphic on $\cc$. The converse statement is proved. 


%
%

{\bf Case 2): $l\geq1$ and there exists a  
solution of Heun equation (\ref{heun}) holomorphic on $\cc^*$ with  $a_{k}\neq0$ for some $k\in\{-l,\dots,-1\}$.} 
In this case the three-diagonal matrix $\wt H$  of relations (\ref{b=0}) with $k=-l,\dots,-1$ has nonzero kernel containing the 
vector $(a_{-1},\dots,a_{-l})$ (see the above arguments), and hence, zero determinant. 
Thus, the matrix $H+\la Id$, whose transposed  is conjugated to $\wt H$ (Proposition \ref{trconj}), also has zero 
determinant. Therefore, equation (\ref{heun2}) has a polynomial solution (Theorem \ref{tpol}). It is known that if the parameters 
$\la$ and $\mu$ are real, $\mu>0$ and equation (\ref{heun2}) has 
a polynomial solution, then the corresponding equation (\ref{heun})  does not have entire solution \cite[theorem 3.10]{bg}: 
cases 1) and 2) are incompatible. Therefore,  if 
$\la,\mu\in\rr$, $\mu>0$ and case 2) takes place, then the monodromy is non-trivial:  it is a unipotent Jordan cell, 

Let us now prove the converse: each statement 1) or 2) implies the existence of a solution $E(z)$ holomorphic on $\cc^*$ of equation 
(\ref{heun}). For statement 1) this is obvious: the solution from 1) is even holomorphic on $\cc$. Let statement 2) hold: equation 
(\ref{heun2}) have a polynomial solution $\wt E$. Let $\Diamond$ be the transformation from the beginning of the subsection, which 
is an  isomorphism between the solution spaces of equations 
(\ref{heun}) and (\ref{heun2}). The function $E=\Diamond^{-1}\wt E$ is a solution of equation 
(\ref{heun}) holomorphic on $\cc^*$, since  $\Diamond$ is an isomorphism of the space of functions holomorphic on 
$\cc^*$. Theorem \ref{talt} is proved.
\end{proof}

 To show that statements 1) and 2) of Theorem \ref{talt} can be compatible, let us recall the definition of modified Bessel functions 
 $I_k(x)$ of the first kind: they are the Laurent coefficients of the family of analytic functions 
 $$g_x(z)=e^{\frac x2(z+\frac 1z)}=\sum_{k=-\infty}^{+\infty}I_k(x)z^k;$$
 $$I_k(x)=i^{-k}J_k(ix),$$
 where $J_k$ is the usual $k$-th Bessel function. 
 Recall that each function $J_k$ of complex variable $x$ has infinite number of zeros (roots), 
 all of them are real (non-zero for $k=0$) and symmetric: that is, if $x$ is a root of the function $J_k$, then 
 so is $-x$. This follows from their infinite product decomposition, see \cite[p.235]{olver}: 
 $$J_k(z)=\frac1{k!}\left(\frac z2\right)^k\prod_{j\geq1}(1-\frac{z^2}{x_{k,j}^2}) \text{ for } k\geq0,$$
 where $x_{k,1}<x_{k,2},\dots\in\rr_+$ are the positive roots of the function $J_k$. 
  This implies that each function $I_k$ has infinite  number of roots, namely, $\pm ix_{k,j}$ (and the additional root $0$, if $k\neq0$).

 \begin{proposition} \label{pbes} 
 Let $l=1$, $n=l+1=2$. Heun equation (\ref{heun2}) has a polynomial solution, if and only if $\la=0$: the latter 
 solution is a constant. For $\la=0$ equation (\ref{heun}) has an entire solution, if and only if $I_1(2\mu)=0$, i.e., 
  $\mu=\pm i\frac{x_{1,j}}2$ for some $j$. Thus, for $l=1$, $n=l+1=2$, $(\la,\mu)=(0,\pm i\frac{x_{1,j}}2)$ equation (\ref{heun2}) has a constant solution  and equation (\ref{heun}) has an entire solution. 
  \end{proposition} 
 
 \begin{proof} The first statement of the proposition follows from Theorem \ref{tpol}: the corresponding three-diagonal matrix for 
 $l=1$ is the scalar number $\la$. Note that a polynomial solution of equation (\ref{heun2}) has degree at most $l-1$, as was 
 shown in \cite[section 3]{bt0}. Hence, in our case it is a constant, which can be normalized to be $\hat E(z)\equiv 1$. Its 
 image $E=\Diamond\hat E=e^{\mu(z+\frac1z)}$ under the transformation $\Diamond$ from the beginning of the subsection is a solution 
 of equation (\ref{heun}). The Laurent coefficient of the function $E(z)$ at the power $z^{-1}$ equals $I_1(2\mu)$, by definition. 
 If it vanishes, then equation (\ref{heun}) has an entire solution, as in the  above proof of Theorem \ref{talt}, case 1). If it does 
 not vanish, then (\ref{heun}) has no entire solution, as was mentioned at the same place, see  
 \cite[lemma 3, statement 6]{bt2}. This proves the proposition. 
\end{proof}

\section{Applications to phase-lock areas in the model of Josephson effect}

\subsection{Phase-lock areas in  Josephson effect:  brief geometric description and plan of the section}

Here we apply the above results to the family of nonlinear equations (\ref{josbeg}): 
 \begin{equation}\dot \phi=\frac{d\phi}{dt}=-\sin \phi + B + A \cos\omega t, \ A,\omega>0, \ B\geq0.\label{jos}\end{equation}
We fix an arbitrary $\omega>0$ and consider family (\ref{jos}) depending on two variable parameters $(B,A)$. The variable change 
$\tau=\omega t$ 
transforms (\ref{jos}) to differential equation (\ref{jostor}) on the two-torus $\mathbb T^2=S^1\times S^1$ with coordinates 
$(\phi,\tau)\in\rr^2\slash2\pi\zz^2$. Its solutions are tangent to the vector field 
\begin{equation}\begin{cases} & \dot\phi=-\frac{\sin \phi}{\omega} + l + 2\mu \cos \tau\\
& \dot \tau=1\end{cases}, \ \ l=\frac B{\omega}, \ \mu=\frac A{2\omega}\label{josvect}\end{equation}
on the torus. The {\it rotation number} of the equation (\ref{jos}) is, by definition, the rotation number 
of the flow of the field (\ref{josvect}), see \cite[p. 104]{arn}. It is a function $\rho(B,A)$ of the parameters: 
$$\rho(B,A;\omega)=\lim_{k\to+\infty}\frac{\phi(2\pi k)}k.$$
Here $\phi(\tau)$ is a general solution of the first equation of system (\ref{josvec}), whose parameter is the initial condition at 
$\tau=0$. Recall that the rotation number does not depend on the choice of the initial condition \cite[p. 104]{arn}. 

We consider the $B$-axis as the {\it abscissa,} and the $A$-axis as  the {\it ordinate.}

\begin{definition} (cf. \cite[definition 1.1]{4}) The {\it $r$-th phase-lock area} is the level set 
$\{(B,A) \ | \ \rho(B,A)=r\}\subset\rr^2$, provided it has a non-empty interior. 
\end{definition}

\begin{remark}{\bf: phase-lock areas and Arnold tongues.} The behavior of phase-lock areas for small $A$ demonstrates the 
effect discovered by V.I. Arnold \cite[p. 110]{arn}. That is why the phase-lock areas became 
 ``Arnold tongues'', see  \cite[definition 1.1]{4}. 
\end{remark}

Recall that the rotation number of equation (\ref{jos}) has the physical meaning of the mean voltage over a long
time interval. The segments in which the phase-lock areas intersect horizontal lines correspond to
the Shapiro steps on the voltage-current characteristic \cite{bktje}.

It has been shown earlier that

- the phase-lock areas  exist only for  integer values of the rotation number (a ``quantization effect'' observed in \cite{buch2} and later also proved 
in  \cite{IRF, LSh2009}); 

- the boundary of each phase-lock area $\{\rho = r\}$ consists of two analytic curves, which are the graphs of two
functions $B=g_{r,\pm}(A)$ (see \cite{buch1}; this fact was later explained by A.V.Klimenko via symmetry, see \cite{RK}); 

-  the latter functions have Bessel asymptotics (observed and proved on physical level in ~\cite{shap}, see also \cite[chapter 5]{lich},
 \cite[section 11.1]{bar}, ~\cite{buch2006}; proved mathematically in ~\cite{RK}).
 
 - each phase-lock area is an infinite chain of bounded domains going to infinity in the vertical direction, in this chain 
 each two subsequent domains are separated by one point. Those of these points that lie in the horizontal $B$-axis are 
 calculated explicitly,  and 
 we call them {\it exceptional}. The other separation points,  lie outside the horizontal $B$-axis and are called the 
 {\it adjacency points} (or briefly {\it adjacencies)}, see Fig.1.
 
  In the present section we obtain functional equations satisfied by non-integer level curves $\{\rho(B,A)=r\}$ of the rotation number 
 (Subsection 5.4) and the boundaries of the phase-lock areas  (Subsection 5.5) 
 using relation of equation (\ref{jos}) to Heun equation (\ref{heun}) (recalled in Subsection 5.3) 
 and the results on monodromy eigenvalues of Heun equations  from the previous section. The above-mentioned functional equations will be written in the complement to the  adjacencies and 
 the discriminant set (with fixed $\omega$) of the parameter values corresponding to the existence of a polynomial solution of equation (\ref{heun2}). 
 Afterwards we discuss open problems and possible approaches to them using the above-mentioned results on Heun equations. 

Symmetries of the phase-lock area portrait are presented in the next subsection.

 \subsection{Symmetries and the portraits of the phase-lock areas}

 It is known that 
 
 - for every $r\in\zz$ the $r$-th phase-lock area is symmetric to the $-r$-th one with respect to the vertical $A$-axis;
  
 - every phase-lock area is symmetric with respect to the horizontal $B$-axis. 
 
 These symmetry statements follow the fact that the transformations $(\phi,\tau)\mapsto(-\phi,\tau+\pi)$ 
 $(\phi,\tau)\mapsto(\phi,\tau+\pi)$ send system (\ref{josvect}) to the system of the same form, 
 where $B$ is changed to $-B$ in the first case and $A$ is changed to $-A$ in the second case. 

In what follows we present pictures of the phase-lock area portraits. Taking into account the above symmetries, it is enough to 
present only the parts of portraits lying in the upper half-plane. 
  
\begin{figure}[ht]
  \begin{center}
   \epsfig{file=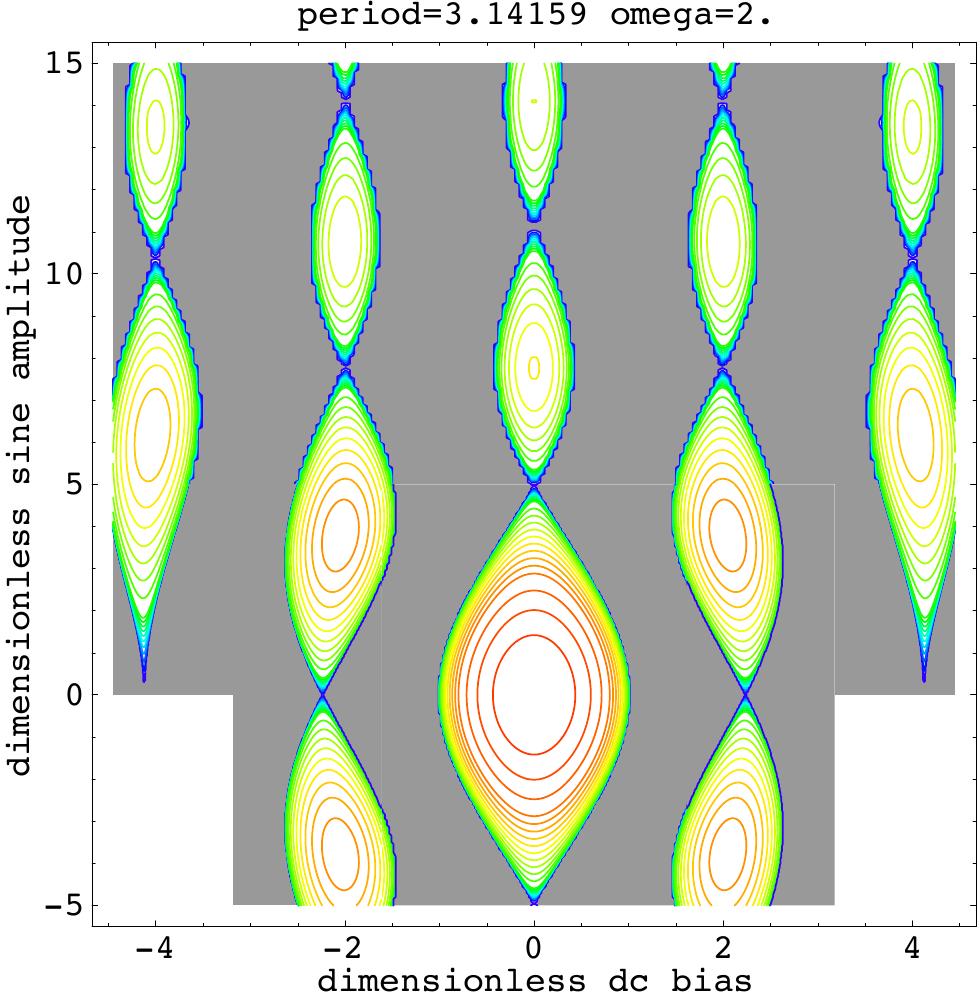}
    \caption{Phase-lock areas and their adjacencies for $\omega=2$. The abscissa is $B$, the ordinate is $A$.}
  \end{center}
\end{figure} 

\begin{figure}[ht]
  \begin{center}
   \epsfig{file=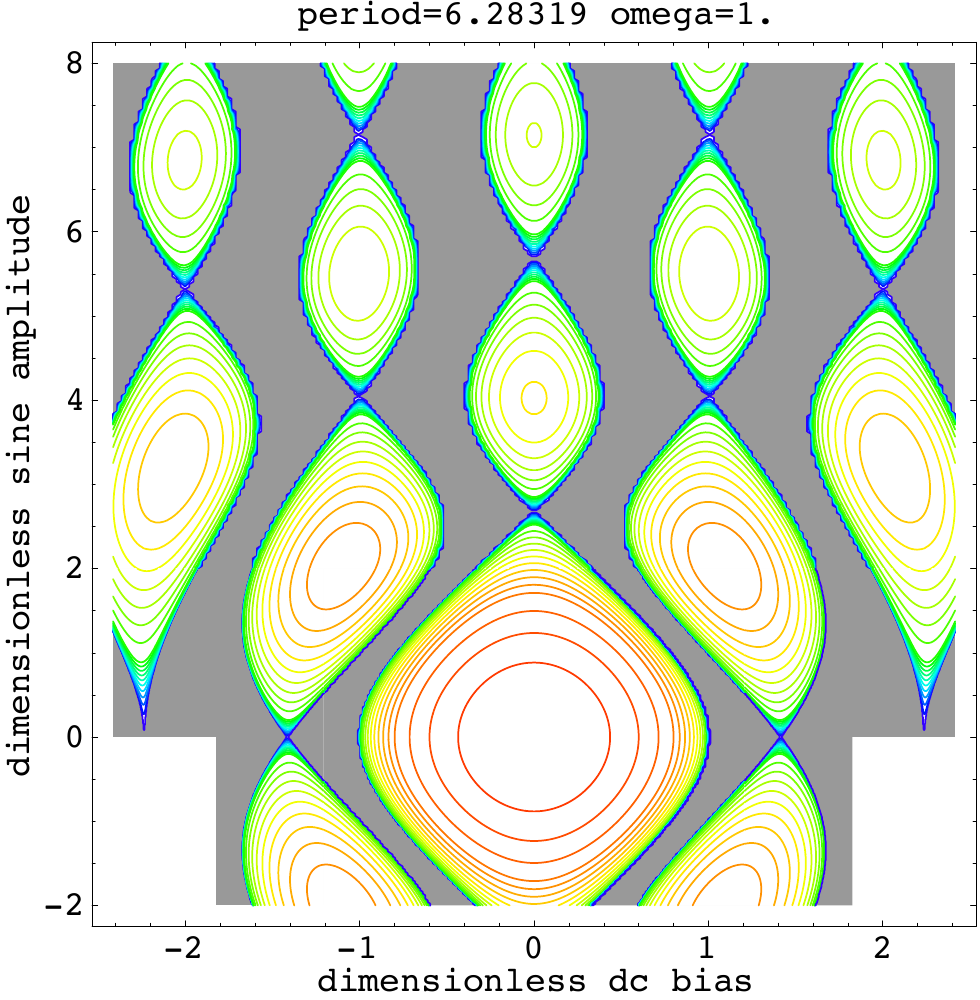}
    \caption{Phase-lock areas and their adjacencies  for $\omega=1$. The abscissa is $B$, the ordinate is $A$.}
  \end{center}
\end{figure} 
 
 \begin{figure}[ht]
  \begin{center}
   \epsfig{file=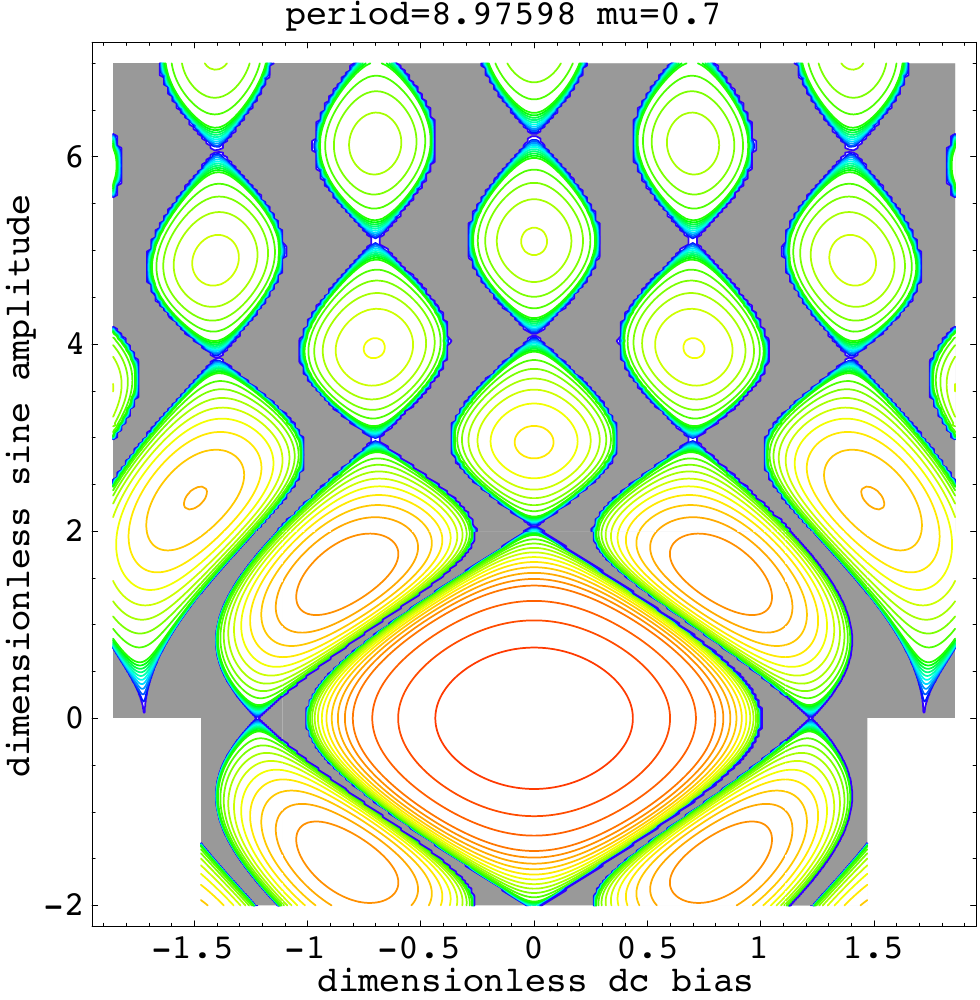}
    \caption{Phase-lock areas and their adjacencies  for $\omega=0.7$. Figure taken from \cite[p. 331]{bt1}.}
  \end{center}
\end{figure} 

 \begin{figure}[ht]
  \begin{center}
   \epsfig{file=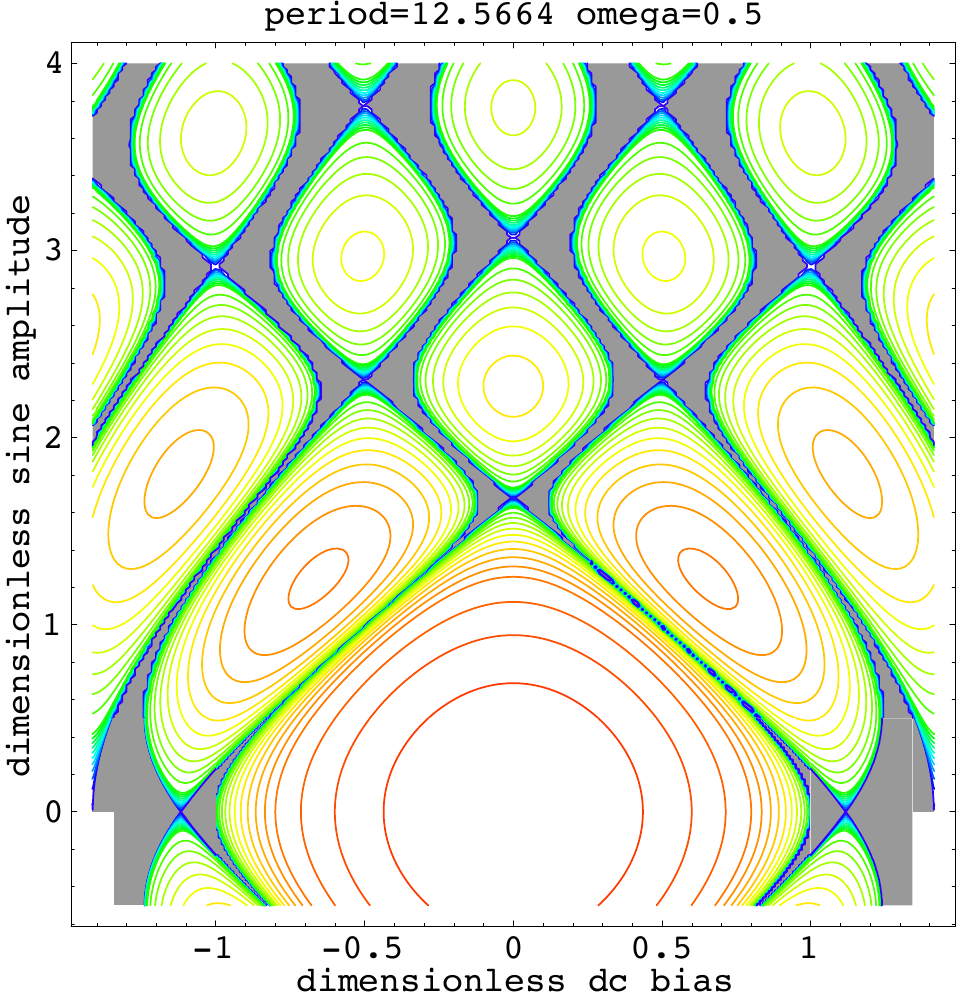}
    \caption{Phase-lock areas and their adjacencies  for $\omega=0.5$.}
  \end{center}
\end{figure} 

\begin{figure}[ht]
  \begin{center}
   \epsfig{file=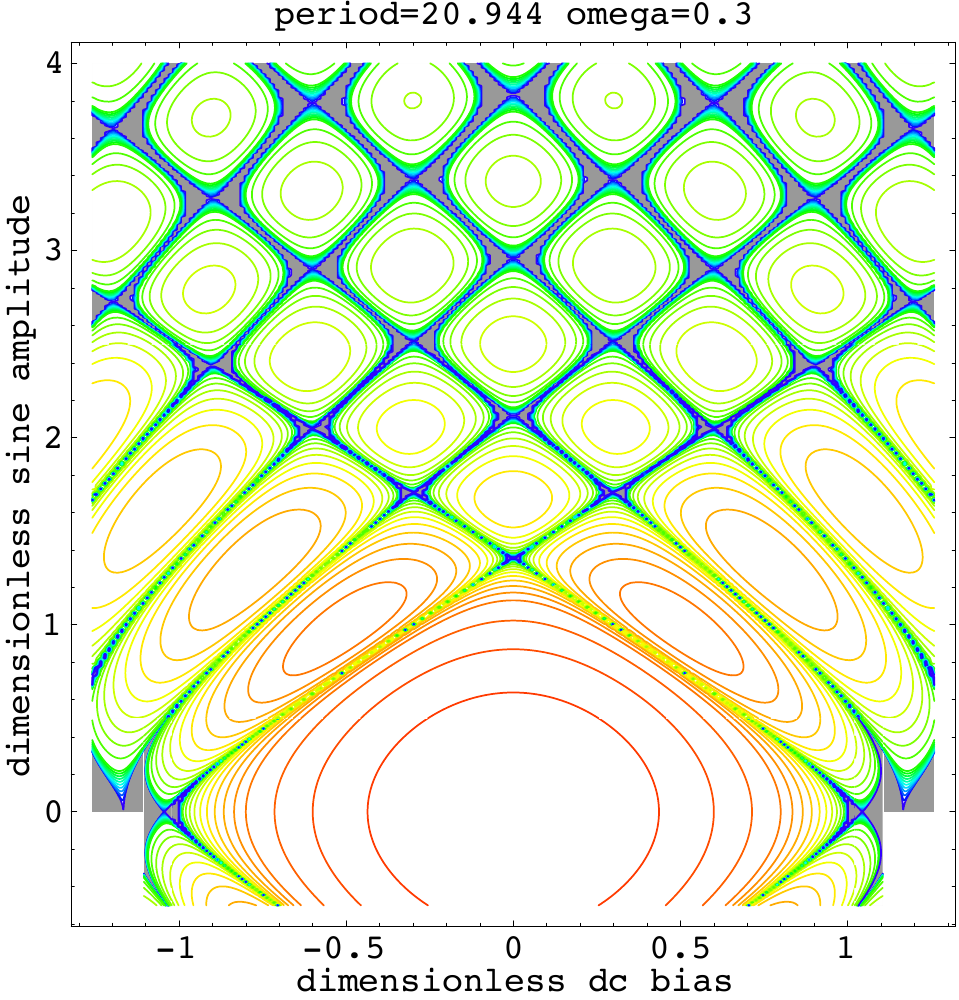}
    \caption{Phase-lock areas and their adjacencies  for $\omega=0.3$.}
  \end{center}
\end{figure} 

 \subsection{Transformation 
 to  double confluent Heun equations. The boundary points of the phase-lock areas corresponding to entire and polynomial solutions 
 of the Heun equation} 

Set  
\begin{equation}\Phi=e^{i\phi}, \ z=e^{i\tau}=e^{i\omega t}.\label{phit}\end{equation}
Considering equation (\ref{jos}) with complex time $t$ we get that transformation (\ref{phit}) sends it to the Riccati equation 
$$\frac{d\Phi}{dz}=z^{-2}((lz+\mu(z^2+1))\Phi-\frac z{2i\omega}(\Phi^2-1)).$$
This equation is the projectivization of the following linear equation in vector function $(u,v)$ with $\Phi=\frac v{u}$, see 
\cite[subsection 3.2]{bg}:
\begin{equation}\begin{cases} & v'=\frac1{2i\omega z}u\\
& u'=z^{-2}(-(lz+\mu(1+z^2))u+\frac z{2i\omega}v)\end{cases}\label{tty}\end{equation}
The above-described reduction to a system of linear equations was earlier obtained in slightly different terms in \cite{bkt1, Foote, bt1, IRF}. 
It is easy to check that a function $v(z)$ is the component of a solution of system (\ref{tty}), if and only if the function $E(z)=e^{\mu z}v(z)$ 
satisfies equation (\ref{heun}) with $n=l+1$ and 
\begin{equation} \lambda=\left(\frac1{2\omega}\right)^2-\mu^2.\label{param}\end{equation} 
The next fact has important applications to the problems discussed in the paper.

\begin{proposition} \label{can} The rotation number function is a
 real-analytic function of the parameters $(B,A)$ on the set $U=\rho^{-1}(\rr\setminus\zz)$, which is the 
 complement to the union of the phase-lock areas of system (\ref{josvect}). Moreover, its restriction to $U$ is an analytic submersion 
 $U\to\rr\setminus\zz$ inducing an analytic fibration by curves 
 $$L_r=\{(B,A)\in\rr^2\ | \ \rho(B,A)=r\},$$
 and these curves are graphs of analytic functions in the variable $A$ defined for all $A\in \rr$. 
\end{proposition}

\begin{proof} The variable change (\ref{phit}) reducing family of equations (\ref{jos})  on the torus to the Riccati equation sends 
the space circle with the coordinate $\phi$  to the unit circle in the Riemann sphere with the coordinate $\Phi$. 
The time $2\pi$ flow mapping (i.e., the first return mapping) 
of the corresponding family of vector fields (\ref{josvect}) on the torus is the restriction 
to the unit circle in the Riemann sphere of a transformation from the M\"obius group: the monodromy transformation of the Riccati equation. 
That is, the restriction to the unit circle of  a conformal automorphism of the unit disk. The rotation number function  $\rho$ 
considered as a function on the group $Aut(D_1)\simeq PSL_2(\rr)$ is analytic on the set of  transformations 
analytically conjugated  to nontrivial rotations. (Recall that these transformations are called {\it elliptic.}). 
This follows from the fact that each elliptic transformation of the unit disk has 
a fixed point inside the disk, its multiplier depends analytically on the transformation and equals $e^{2\pi i\rho}$. 
The set of elliptic transformations coincides with the complement of   the subset in $Aut(D_1)$ consisting of 
transformations whose restrictions to $S^1=\partial D_1$ have integer rotation number. 
The points of the set $U$ correspond to non-integer values of the rotation number, and hence, to Riccati equation with 
elliptic monodromy. This proves analyticity of the rotation number function on the set $U$. 

Now let us prove the second statement of the proposition: the restriction $\rho|_U$ is a submersion, and each its level set $L_r$ is 
the graph $\{ B=g_r(A)\}$ of a function $g_r$ analytic on $\rr$. Indeed, let $r\notin\zz$. The set $L_r$ is an analytic curve, by 
analyticity of the rotation number function on $U$ (the first statement of the proposition).  Let us prove that it is a graph of function. 
To do this, let us check  that $\frac{\partial \rho}{\partial B}(B,A)>0$ 
for every $(B,A)\in U$, and hence, one can apply the Implicit Function Theorem to the function $\rho$. 
The time $2\pi$ flow mapping of the corresponding vector field (\ref{josvect}) has a non-integer rotation number, and hence, is analytically 
conjugated to the rotation by angle $2\pi\rho_0$, $\rho_0=\rho(B,A)$, as was shown above. 
Perturbation of the time $2\pi$ flow mapping via  increasing $B$ by $\delta B$ 
 is conjugated to a perturbed rotation, which is the composition of the above rotation with a diffeomorphism $h$
 close to identity, the difference $h(x)-x>0$ being 
 uniformly bounded from below by $c\delta B$, $c>0$. The latter perturbed rotation has rotation number bounded from below 
 by $\rho_0+\frac c{2\pi}\delta B$, which follows from the definition of the rotation number. This implies the above statement on positivity  
 of the derivative in $B$ of the rotation number function. Thus, the mapping $\rho|_U$ is an analytic submersion, and the set $L_r$ 
 with $r\notin\zz$  is the graph of an analytic function $g_r(A)$.  Now it remains to show that the function $g_r$ is analytic on the 
 whole line $\rr$, i.e., the curve $L_r$ does not go to infinity in the horizontal direction so that the ordinate $A$ remains bounded. This 
 follows from the fact that the curve $L_r$ is contained in the domain between the  phase-lock areas with the rotation 
 numbers $[r]$ and $[r]+1$, and the abscissa $B$ is uniformly bounded on this domain. The latter statement follows from the 
 fact that each phase-lock area $\{\rho=l\}$ tends to infinity in the vertical direction approaching the line $\{ B=l\omega\}$: its 
 boundary consists of graphs of two functions $\{ B=g_{l,\pm}(A)\}$ defined on $\rr$ and having
Bessel asymptotics at infinity \cite{RK}. This proves the proposition.
\end{proof}

%
%

\begin{remark} \label{rktriv} 
The adjacencies of the phase-lock areas of family of equation (\ref{jos}) are characterized by the condition that the 
corresponding period $2\pi$ flow mapping (i.e., the 
Poincar\'e mapping) of  vector field from family (\ref{josvect}) is the identity,  see \cite[proposition 2.2]{4}. It was shown in 
\cite[lemmas 3, 4]{bt1} that if a point $(B,A)$ corresponds to a Heun equation (\ref{heun}) with an entire solution, then 
the Poincar\'e mapping corresponding to this point is the 
identity, and hence, it is an adjacency. The complete result is  the following. 
\end{remark}

\begin{theorem} \label{tadj} (see \cite[theorems 3.3, 3.5]{bg}).  For every $\omega>0$, $l\geq0$  a point $(B,A)\in\rr^2$ with $A\geq0$, 
$B=l\omega$ 
 is an adjacency for family of equations (\ref{jos}), if and only if $l\in\zz$ and 
the corresponding equation (\ref{heun}) with  $n=l+1$ and $\mu$, $\la$ as in (\ref{josvect}) and (\ref{param}) has a nontrivial entire solution, i.e., 
if and only if equation $\xi_l(\la,\mu)=0$ holds; $\xi_l$ is the same, as in (\ref{xil}). 
\end{theorem}

%
In what follows we will use the next result. 
\begin{proposition} \label{roteig}  Let  $\omega>0$, $(B,A)\in\rr^2$, and let $\rho=\rho(B,A)$ denote the corresponding rotation number. 
Let $(B,A)$ do not lie in the interior of a phase-lock area: it may lie either in the complement of the union of phase-lock areas, or 
in its boundary. Then  the monodromy operator of the corresponding Heun equation (\ref{heun}) with 
$n=l+1$, $l=\frac B{\omega}$ has eigenvalues $e^{\pi i(\rho-l)}$ and $e^{-\pi i(\rho+l)}$. 
\end{proposition}

\begin{proof} Let $\la_1$, $\la_2$ be the eigenvalues of the above monodromy operator of Heun equation. 
The point $(B,A)$ does not lie in the interior of a phase-lock area. If $r=\rho(B,A)\notin\zz$, then the monodromy of the corresponding Riccati equation 
is an elliptic M\"obius transformation conformally conjugated to the rotation by angle $2\pi r$. Therefore, it has two fixed points with multipliers $e^{\pm2\pi i r}$.   The latter multipliers  are ratios of the eigenvalues $\la_j$, and without loss of generality we consider that $\frac{\la_1}{\la_2}=e^{2\pi i r}$. 
On the other hand, $\la_1\la_2=e^{-2\pi i l}$, by Proposition \ref{monprod}. This implies that the eigenvalues under question are equal to 
$\pm e^{\pi i(r-l)}$, $\pm e^{-\pi i(r+l)}$. In the case, when $r\in\zz$, the point $(B,A)$ lies in the boundary of a phase-lock area and the monodromy of the Riccati equation is either parabolic, or identity. The corresponding 
monodromy of the Heun equation has multiple eigenvalue given by  same 
(now coinciding) formulas.  The correct sign should be the same for all the points $(B,A)$ in the complement of the parameter plane 
to the union of the interiors of phase-lock areas, by path connectivity of the latter complement and continuity of the rotation number 
function. Indeed, the phase-lock areas are disjoint closed subsets in $\rr^2$, the complement of each of them consists of two connected components. The interior of each phase-lock area is disconnected: one can pass from one 
its side to the other via any adjacency. 
The sign is ``$+$'' at each  adjacency, since the corresponding monodromy is trivial (Remark \ref{rktriv}). 
Hence, it is ``$+$'' everywhere. This proves the proposition.
\end{proof}

\begin{corollary} \label{smsl} The pair of eigenvalues from Proposition \ref{roteig}  corresponding to a given point $(B,A)$ is 
the same for all other points $(B,A')$ with  $\rho(B,A')\equiv\pm\rho(B,A)(mod2\zz)$. 
\end{corollary}

\begin{theorem} \label{poteig} Let $\omega>0$, $(B,A)\in\rr^2$, $B,A>0$, $l=\frac B\omega$, $\mu=\frac A{2\omega}$, 
$\la=\frac1{4\omega^2}-\mu^2$, $\rho=\rho(B,A)$. The double confluent Heun equation (\ref{heun2}) corresponding to the latter 
$\la$, $\mu$ and $l$ has a polynomial solution, if and only if $l,\rho\in\zz$, $0\leq\rho\equiv l(mod 2\zz)$ and $\rho\leq l$, and in addition, 
the point $(B,A)$ lies in the boundary of the phase-lock area number $\rho$ and is not an adjacency. In other terms, 
the points $(B,A)\in\rr_+^2$ corresponding to equations (\ref{heun2}) with polynomial solutions lie in boundaries of 
phase-lock areas and  are exactly their intersection points with the lines $l=\frac B\omega\equiv\rho(mod 2\zz)$, $0\leq\rho\leq l$, except for 
the adjacencies, see Fig. 6. 
\end{theorem}

 \begin{figure}[ht]
  \begin{center}
   \epsfig{file=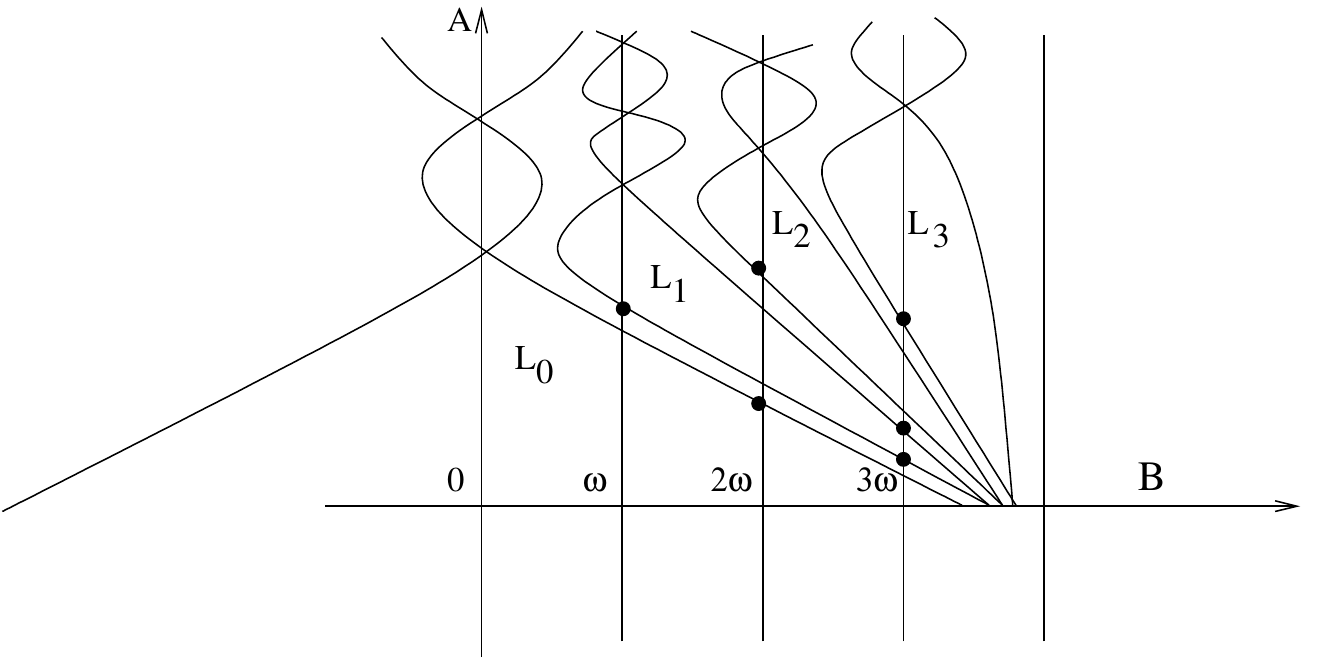}
    \caption{Approximate phase-lock areas for $\omega\simeq 0.27$; 
    the marked points correspond to equations (\ref{heun2}) with polynomial solutions.}
  \end{center}
\end{figure} 

\begin{proof} It is known  that every point $(B,A)\in\rr_+^2$ corresponding to equation (\ref{heun2}) with a polynomial solution 
lies in the boundary of the phase-lock area number $\rho$, and one has $l,\rho\in\zz$, 
$l=\frac B\omega\equiv\rho(mod 2\zz)$, $0\leq\rho\leq l$ \cite[corollary 6 and theorem 5]{bt0}. In addition, $(B,A)$ is not an adjacency 
\cite[theorem 3.10]{bg}. 
Let us prove the converse: if $(B,A)$ satisfy all the latter statements, then 
the corresponding equation (\ref{heun2}) has a polynomial solution. Indeed, if $(B,A)\in\rr^2_+$ lies in the boundary of the phase-lock area number $\rho$, $l\in\zz$ and $\rho\equiv l(mod 2\zz)$, then 
the monodromy of  Heun equation (\ref{heun})  is unipotent, by Proposition \ref{roteig}: the corresponding 
eigenvalues are equal to $e^{-\pi i(\rho+l)}=1$. Let us now suppose that $(B,A)$ is not an adjacency, or equivalently, 
equation (\ref{heun}) does not have an entire solution. Then equation (\ref{heun2}) 
has a polynomial solution, by Theorem \ref{talt}. Theorem \ref{poteig} is proved.
\end{proof} 

\subsection{Equation on non-integer level sets of rotation number}

Recall that in the classical Arnold family of circle diffeomorphisms the phase-lock areas (i.e., the Arnold tongues) exist exactly for the 
rational values of the rotation number. In  family (\ref{jos})  we have 
the quantization effect mentioned in the Subsection 5.1: phase-lock areas exist only for integer values of the rotation number. 
Therefore, in our case the rational non-integer values of the rotation number do not differ from the irrational values. 

Recall that $L_r=\{(B,A)\in\rr^2 \ | \ \rho(B,A)=r\}$, and for $r\notin\zz$ this is the graph $\{ B=g_r(A)\}$ of an analytic 
function (Proposition \ref{can}).  For given $\omega>0$ and $r>0$ set 
$$L_{<r>}=\sqcup_{v\equiv\pm r(mod 2\zz)}L_{v}.$$ 

\begin{remark} Our goal is to describe  the sets $L_r$ for $r\notin\zz$ by analytic equations in terms of the monodromy eigenvalues of 
Heun equations, by using their expressions $e^{\pi i(\pm\rho-l)}$ via the rotation number function (Proposition \ref{roteig}). 
The latter pair of eigenvalues is in a one-to-one correspondence with the value $\pm\rho(mod 2\zz)$. Therefore, 
the pair of eigenvalues corresponding to a given value $r$ of the rotation number function coincides with all the other pairs 
corresponding  to the values $r'\equiv\pm r(mod 2\zz)$, see Corollary \ref{smsl}. Our methods allows to describe 
 the union $L_{<r>}$ by analytic equation, and not each individual level set $L_r$ separately. 
\end{remark}

We consider the case, when $r\notin\zz$. Let us write down analytic equations defining the set $L_{<r>}$ in the complement 
$\rr^2\setminus\Sigma_{<r>}$, where 
 \begin{equation} \Sigma_{<r>}=\cup_{\pm}\{ (B,A)\in\rr^2 \ | \ l=\frac B{\omega}\equiv\pm r(mod 2\zz)\}.\label{neqm2}\end{equation}
For every $(B,A)\in L_{<r>}$ the corresponding 
 Heun equation (\ref{heun}) has a monodromy eigenfunction of the type 
$$E(z)=z^b\sum_{k\in\zz}a_kz^k, \ b=\frac{r-l}2, \ l=\frac B{\omega},$$
 by Proposition \ref{roteig}. On has $b, b+l\notin\zz$, if $(B,A)\notin\Sigma_{<r>}$. Therefore, the analytic subset  $L_{<r>}\setminus\Sigma_{<r>}\subset(\rr^2\setminus\Sigma_{<r>})$ 
  is described by equation (\ref{paste}). Let us write it down explicitly. The corresponding matrices $M_k$, $R_k$, $S_m$, $T_m$, 
  see (\ref{mat3}) and  (\ref{defsm}) have the form 
 $$M_k=\left(\begin{matrix}  1+\frac{\la}{(k+\frac r2)^2-\frac{l^2}4} & \frac{\mu^2}{(k+\frac r2)^2-\frac{l^2}4}\\ 
  1 & 0\end{matrix}\right), \ R_k=M_kM_{k+1}\dots,$$
 $$S_m=\left(\begin{matrix}  1+\frac{\la-l+1}{(\frac{r-l}2-m+1)(\frac{r+l}2-m-1)} & 
 \frac{\mu^2(\frac{r+l}2-m)}{(\frac{r-l}2-m+1)(\frac{r+l}2-m-1)(\frac{r+l}2-m-2)}\\
  1 & 0\end{matrix}\right),$$ 
 $$T_m=S_mS_{m+1}\dots.$$
 \begin{theorem} \label{rhonon}  Let $\omega>0$, $r\in\rr$, $r\notin\zz$.  The set $L_{<r>}\cap(\rr^2\setminus \Sigma_{<r>})$ 
 is defined by the following equation in the variables $(B,A)\in\rr^2\setminus \Sigma_{<r>}$, where $B=l\omega$, $A=2\mu$, 
 $\la+\mu^2=\frac1{4\omega^2}$: 
 \begin{equation} (\frac{r-l}2+1)(\frac{r+l}2-1)R_{1,11}T_{0,11}+\mu^2R_{1,21}T_{0,21}=0.\label{pasterho}\end{equation}
 \end{theorem}

\begin{proof} A point $(B,A)\in\rr^2\setminus \Sigma_{<r>}$ is contained in $L_{<r>}$, if and only if 
$\rho=\rho(A,B)\equiv\pm r(mod 2\zz)$. Or equivalently, some of the corresponding monodromy 
eigenvalues $e^{\pi i(\pm\rho-l)}$ equals $e^{2\pi ib}=e^{\pi i(r-l)}$.  The latter statement is equivalent to (\ref{pasterho}), by 
Theorem \ref{eight} and since $b,b+l\notin\zz$. This proves Theorem \ref{rhonon}.
\end{proof}

\subsection{Description of boundaries of phase-lock areas}

\begin{proposition} \label{pjcell} 
 A point in the parameter space  of equation (\ref{jos}) lies in the boundary of a phase-lock area, if and only if the monodromy of the 
corresponding Heun equation (\ref{heun}) is parabolic: has multiple eigenvalue.
\end{proposition}

\begin{proof} The point under question lyes in the boundary of a phase-lock area, if and only if the flow mapping of the vector field (\ref{josvect}) for the period $2\pi$ (restricted to the coordinate $\phi$-circle) is parabolic: has a fixed point 
with unit derivative. The period mapping is the restriction to the unit circle of the monodromy of the corresponding Riccati equation: 
the projectivized monodromy. Parabolicity of the projectivization of a two-dimensional linear operator is equivalent to its own 
parabolicity. The proposition is proved.
\end{proof}
 
 \begin{proposition} \label{pjcell2} Let  a Heun equation (\ref{heun})  have a parabolic monodromy. Then the monodromy either has Jordan cell type, or is  the identity. 
\end{proposition}

\begin{proof} The monodromy matrix is the product of the formal monodromy matrix $\diag(1,e^{-2\pi i n})$ and a pair of unipotent 
matrices: the inverse to the Stokes matrices at 0, cf. \cite[formulas (2.15) and (3.2)]{4}.   If the monodromy of 
a Heun equation (\ref{heun}) is parabolic but not a Jordan cell, then it is a multiplication by scalar number. If the 
monodromy is  scalar, 
then the Stokes matrices are trivial, and the monodromy coincides with the formal one, see \cite[proof of lemma 3.3]{4}. 
Hence, both monodromies are scalar and  given by the above diagonal matrix with unit eigenvalue. 
Thus, they are trivial. The proposition is proved.
\end{proof}


%

The condition saying that the monodromy has multiple eigenvalue 
is equivalent to the statement that it has eigenvalue $\pm e^{-\pi il}$, 
by Proposition \ref{monprod}. This is equivalent to the statement that there exists a multivalued solution $z^b\sum_{k\in\zz}a_kz^k$ 
of Heun equation with $b\in\{ -\frac l2, \ -\frac{l+1}2\}$: a monodromy eigenfunction with the above eigenvalue. The corresponding 
parameter set of Heun equations (\ref{heun}) for $b\notin\zz$ 
 will be described below (Cases 1 and 2) by using the following proposition. Afterwards we 
 immediately  obtain the description of boundaries of phase-lock areas.

\begin{proposition} \label{dinvar} Let a Heun equation (\ref{heun}) have a Jordan cell monodromy. Then its eigenfunction is  either invariant, or anti-invariant under the involution $\#$: 
$$(\# E)(z)=2\omega z^{-l-1}(E'(z^{-1})-\mu E(z^{-1})).$$
\end{proposition}

The proposition follows from the fact that the involution $\#$ sends monodromy eigenfunctions to eigenfunctions (Proposition \ref{propdiez}). 



{\bf Case 1: $l\notin 2\zz$, $b=-\frac l2$: the monodromy eigenvalue equals $e^{-\pi i l}\neq1$.}  
Then the monodromy operator of Heun equation (\ref{heun}) is a Jordan cell, by Proposition \ref{pjcell2}.  
Consider the matrices 
\begin{equation}M_k=\left(\begin{matrix}  1+\frac{\la}{k^2-\frac{l^2}4} & \frac{\mu^2}{k^2-\frac{l^2}4}\\  1 & 0\end{matrix}
\right), \ R_k=M_kM_{k+1}\dots.\label{rkmk2}\end{equation}
\begin{theorem} \label{heun4} Let  $\la+\mu^2=\frac1{4\omega^2}$, $\omega,\mu\neq0$, $n=l+1$, $l\notin 2\zz$.  
The monodromy operator of the 
corresponding Heun equation (\ref{heun}) has eigenvalue $e^{-\pi il}$, if and only if 
\begin{equation} R_{0,21}\pm\omega l(R_{0,21}-R_{0,11})=0.\label{r0eql}\end{equation}
with some choice of sign.
\end{theorem}

\begin{proof} The corresponding eigenfunction $E$ has the form 
\begin{equation} E(z)=z^{-\frac l2}f(z), \ f(z)=\sum_{k\in\zz}a_kz^k \text{ is holomorphic on } \cc^*.\label{eheun}\end{equation}
Equation (\ref{heun}) is equivalent to recurrence equations (\ref{recur}) with $b=-\frac l2$:  
\begin{equation} (k^2-\frac{l^2}4+\la) a_k-\mu(k+\frac l2)a_{k-1}+\mu(k-\frac l2+1)a_{k+1}=0.\label{l2}\end{equation}
The series $f(z)$ should converge on $\cc^*$. The above matrices $M_k$ and $R_k$ coincide with those constructed in 
(\ref{mat3}), and they are well-defined for all  $k\in\zz$. Therefore, the coefficients $a_k$, $k\geq0$ are given by formulas (\ref{ak}) 
up to common constant factor, by Theorems \ref{conv} and \ref{th24}: 
\begin{equation} a_0=-\frac2lR_{0,21}, \ a_{-1}=\mu^{-1}R_{0,11}.\label{a0l1}\end{equation}

Now we will use the  condition of (anti-) invariance $\# E=\pm E$ (Proposition \ref{dinvar}), which takes the form
$$\sum_{k\in\zz}a_kz^{k-\frac l2}=\pm2\omega(\sum_{k\in\zz}(k-\frac l2)a_kz^{-k-\frac l2}-\mu\sum_{k\in\zz}a_kz^{-\frac l2-k-1}),$$
or equivalently, 
\begin{equation}\sum a_kz^k=\pm2\omega(\sum(k-\frac l2)a_kz^{-k}-\mu\sum a_kz^{-k-1}).\label{dinv}\end{equation}
The free (zero power) term of the latter equation is equivalent to the relation 
\begin{equation} (1\pm l\omega)a_0\pm2\omega\mu a_{-1}=0,\label{eql}\end{equation}
which is in its turn  equivalent to (\ref{r0eql}), by (\ref{a0l1}). Therefore, existence of the above solution $E$ implies (\ref{r0eql}). 

Let us prove the converse: each equation (\ref{r0eql}) implies the existence of a solution (\ref{eheun}) of  Heun equation. 
 To do this, consider the action of the 
transformation $\#$  on the {\it formal} series (\ref{eheun}) (with $f$ not necessarily converging). It  sends formal 
solutions of Heun equation (equivalently, formal solutions of (\ref{l2})) to formal solutions. (The proof 
of  symmetry of Heun equation under the transformation $\#$ uses only Leibniz differentiation rule and remains valid for formal series.) 
The space of formal solutions is two-dimensional, and it is identified with the space of its initial conditions $(a_{-1},a_0)$. The 
transformation $\#$ is its involution. Its eigenvalues are equal to $\pm1$, and the corresponding 
eigenspaces are defined by initial conditions that satisfy (\ref{eql}). Therefore, both eigenspaces are one-dimensional and are exactly 
characterized by equations (\ref{eql}), since both equations (\ref{eql}) are nontrivial. Thus, {\it a formal solution $(a_k)_{k\in\zz}$ of 
recurrence relations (\ref{l2}) is $\#$-(anti)-invariant, if and only if its coefficients $a_{-1}$, $a_0$ satisfy (\ref{eql}) with the corresponding 
sign.}

Fix the one-sided   solution $\sum_{k\geq-1}a_kz^k$ of recurrence relations (\ref{l2}) for $k\geq0$. It satisfies (\ref{eql}), by 
(\ref{r0eql}). The sequence $(a_k)_{k\geq-1}$ extends uniquely  to a two-sided formal solution $(a_k)_{k\in\zz}$  of (\ref{l2}) (a priori, not necessarily presenting a converging series for $k\to-\infty$), since the coefficients at $a_{k\pm1}$ in (\ref{l2}) do not vanish. The 
latter formal solution should be $\#$-(anti-) invariant, by (\ref{eql}) and the previous statement. Hence, 
 $$a_k=\pm2\omega((-k-\frac l2)a_{-k}-\mu a_{-(k+1)})$$
 by (\ref{dinv}). The series $\sum_{k<0}a_kz^k$ converges on $\cc^*$, since it is bounded from above by converging series 
 $2\omega\sum_{k\geq0}(k+|l|+\mu+1)|a_{k}z^k|$, by the latter formula. This together with the above argument proves the theorem.
 \end{proof}  

{\bf Case 2: $l\notin 2\zz+1$ and $b=-\frac{l+1}2$: the monodromy eigenvalue equals $-e^{-\pi i l}\neq1$.} 
Then the monodromy of equation (\ref{heun}) is a Jordan cell, as above.  
Consider the matrices 
\begin{equation}M_k=\left(\begin{matrix}  1+\frac{\la}{(k-\frac12)^2-\frac{l^2}4} & \frac{\mu^2}{(k-\frac12)^2-\frac{l^2}4} \\  1 & 0\end{matrix}\right), \ 
R_k=M_kM_{k+1}\dots.\label{mkrk1}\end{equation}

\begin{theorem} \label{heun5} Let  $\la+\mu^2=\frac1{4\omega^2}$, $\omega,\mu\neq0$, $n=l+1$, $l\notin 2\zz+1$.  The monodromy operator of the 
corresponding Heun equation (\ref{heun}) has eigenvalue $-e^{-\pi il}$, if and only if 
\begin{equation}  R_{1,11}\pm2\omega\mu(R_{1,11}-R_{1,21})=0\label{r1eql}\end{equation}
with some choice of sign. 
\end{theorem}

\begin{proof} 
We are looking for a double-infinite  solution of Heun equation  (\ref{heun}) of the type 
\begin{equation} E(z)=z^{-\frac{l+1}2}f(z), \ f(z)=\sum_{k\in\zz}a_kz^k\label{eheun1}\end{equation}
 of Heun equation (\ref{heun}) with $f$ holomorphic on $\cc^*$.  That is, with $a_k$ satisfying 
 recurrence relations (\ref{recur}) for $b=-\frac{l+1}2$, which take the form 
\begin{equation} ((k-\frac12)^2-\frac{l^2}4+\la)a_k-\mu(k+\frac{l-1}2)a_{k-1}+\mu(k-\frac{l-1}2)a_{k+1}=0.\label{l+12}\end{equation}
The above matrices $M_k$ and $R_k$ coincide with those constructed in (\ref{mat3}), and they 
are well-defined for all  $k\in\zz$. Therefore, the coefficients $a_k$, $k\geq0$ are given by formulas (\ref{ak}) 
up to common constant factor, by Theorems \ref{conv} and \ref{th24}. In particular, 
$$a_0=\frac1bR_{0,21}=\frac1bR_{1,11}=-\frac 2{l+1}R_{1,11}, \ a_1=\frac{\mu}{b(b+1)}R_{1,21}=\frac{4\mu}{l^2-1}R_{1,21}.$$
The condition of (anti-) invariance under the involution $\#$ of the solution takes the form 
$$\sum a_kz^{k-\frac{l+1}2}=\pm2\omega(\sum(k-\frac{l+1}2)a_kz^{-k-\frac{l-1}2}-\mu\sum a_kz^{-k-\frac{l+1}2}),$$
or equivalently, 
$$\sum a_kz^k=\pm2\omega(\sum(k-\frac{l+1}2)a_kz^{-k+1}-\mu\sum a_kz^{-k}).$$
The free term  of the latter equation is given by the  relation 
\begin{equation} (1\pm2\omega\mu)a_0\pm\omega(l-1)a_1=0,\label{l12}\end{equation}
which is equivalent to (\ref{r1eql}). The rest of proof of Theorem \ref{heun5} is analogous to the proof of Theorem \ref{heun4}. 
\end{proof}

%
%

\begin{corollary} \label{cboun} 
Let $\omega,\mu>0$, $\la+\mu^2=\frac1{4\omega^2}$, $l\geq0$, $n=l+1$, $B=l\omega$, $A=2\mu\omega$.  
The point $(B,A)$ lies in the boundary of a phase-lock area, if and only if one of the following four incompatible statements holds:

1)  $(B,A)$ is an adjacency:  $l\in\zz$ and $\xi_l(\la,\mu)=0$; 

2) Heun equation (\ref{heun2}) has a polynomial solution: $l\in\nn$ and  
$\det(H+\la Id)=0$, where $H$ is the $l\times l$-matrix from (\ref{defh});

3) $l\notin2\zz$ and equation (\ref{r0eql}) holds;

4) $l\notin2\zz+1$ and equation (\ref{r1eql}) holds.
\end{corollary}

\begin{proof} If one of the above statements holds, then $(B,A)$ lies in the boundary of a phase-lock area, by Proposition \ref{pjcell} and 
Theorems \ref{talt}, \ref{heun4}, \ref{heun5}.  Conversely, let  $(B,A)$ lie in the boundary of a phase-lock area. The monodromy 
of the corresponding Heun equation (\ref{heun}) is parabolic, by Proposition \ref{pjcell}. If it is uniponent, then $l\in\zz$, since 
its determinant $e^{-2\pi i l}$ should be unit.  If $l\in\zz$, then it is unipotent, if and only if  some of 
the two incompatible statements 1) or 2) holds, by Theorems \ref{talt} and \ref{tadj}. Otherwise, the monodromy has 
Jordan cell type with eigenvalue $\pm e^{-\pi i l}\neq1$, by Propositions \ref{roteig} and \ref{pjcell2}. 
Therefore, one of the statements 3) or 4) holds, by Theorems 
\ref{heun4} and \ref{heun5}. Statements 3) and 4) are incompatible: they correspond to Heun equation (\ref{heun}) with monodromy having 
multiple eigenvalue $e^{-\pi i l}$ or $-e^{-\pi i l}$ respectively. This proves the corollary.
\end{proof}

\begin{proposition} \label{newex} Let $l\in\zz$. For given $\omega,\mu>0$ and $n=l+1$ the corresponding Heun equation (\ref{heun}) has a monodromy eigenfunction with eigenvalue $-1$, if and only if the corresponding point $(B,A)\in\rr^2$ lies in the 
boundary  of a phase-lock area with a rotation number $\rho\equiv l+1(mod 2\zz)$. 
\end{proposition}

\begin{proof} For $l\in\zz$  the monodromy has unit determinant (Proposition \ref{monprod}). Therefore, if it has eigenvalue $-1$, then 
its other eigenvalue is also $-1$. Hence, the point $(B,A)$ lies in the boundary of the phase-lock area number $\rho=\rho(B,A)$ (Proposition \ref{pjcell}). Thus, $e^{\pi i(\pm \rho-l)}=-1$ for both signs, 
 by Proposition \ref{roteig}. The latter equality holds if and only if  $\rho\equiv l+1(mod 2\zz)$.  Conversely, if a point $(B,A)$ with 
 $l=\frac B{\omega}\in\zz$ lies in the boundary of a phase-lock area, and $\rho(B,A)$ satisfied the above equality, then the monodromy 
 eigenvalues are equal to $-1$, by Proposition \ref{roteig}.  The proposition is proved.
\end{proof}


\def\mca{\mathcal A}

\subsection{Conjectures on geometry of phase-lock areas}

Here we state conjectures that are motivated by numerical simulations and theoretical results of \cite{bt0, bt1, bg, 4}. Recall that  
for every $r\in\zz$ we denote 
$$L_r=\text{ the phase-lock area number } r.$$
The next five conjectures are due to the first author (V.M.Buchstaber) and S.I.Tertychnyi. 

\begin{conjecture} \label{garl} 
The upper part $L_r^+=L_r\cap\{ A\geq0\}$ of each phase-lock area $L_r$ 
is a garland of infinitely many connected components separated by adjacencies 
$\mathcal A_{r,1}, \mathcal A_{r,2}\dots$ lying in the line $\{ B=r\omega\}$ and 
ordered by their $A$-coordinates. 
\end{conjecture} 

\begin{remark} \label{nwrem} It was shown in \cite[theorems 1.2, 3.17]{4} that for every $r\in\zz$ the abscissa of each adjacency in $L_r$ 
equals $l\omega$, $l\in\zz$, $l\equiv r(mod 2\zz)$; $0\leq l\leq r$ if $r\geq0$; $r\leq l\leq0$ if $r\leq0$. 
\end{remark}

\begin{conjecture} \label{interv} For every $k\geq2$ the $k$-th component in $L_r^+$ contains the interval 
$(\mathcal A_{r,k-1},\mathcal A_{r,k})$. 
\end{conjecture}

Let us introduce the  function $\eta(P)$ defined on the interior of each phase-lock area $L_r$: the value $\eta(P)$ 
 is defined to be the length of the Shapiro step through the point $P$, that is, the length of the  intersection of the phase-lock area 
 $L_r$ with the horizontal line 
 $\{ A=A(P)\}$. For physical applications it is important to know how to find the 
maxima of the function $\eta(P)$ on the connected components of the interior $Int(L_r)$. 
\begin{problem} Find the maximal value of the function $\eta(P)$ on each connected component of every phase-lock area.
\end{problem}


\begin{problem} Is it true that for every given $k\in\mathbb N$ all the adjacencies $\mca_{r,k}$, $r=1,2,\dots$, 
lie on the same line  with azimuth depending on $k$, see Fig. 1--5? 
\end{problem}

\begin{proposition} The first component of the zero phase-lock area contains the interior of the 
square with vertices $(0,\pm1)$, $(\pm1,0)$.
\end{proposition}

\begin{proof} Let $Q$ denote the interior of the square under question: it is defined by the inequality $|A|+|B|< 1$. 
Let us show that for every $(B,A)\in Q$ the $\phi$-component of vector field (\ref{josvect}) is negative whenever 
$\phi=\frac{\pi}2$ and positive whenever  $\phi=-\frac{\pi}2$. This implies that  its  flow map  for any time sends the space segment 
$[-\frac{\pi}2,\frac{\pi}2]$ strictly to itself and hence, has a fixed point there, and thus, has zero  rotation number. Indeed,  
the $\phi$-component of vector field (\ref{josvect}) equals 
$$\dot\phi=\frac1{\omega}(-\sin\phi+B+A\cos\tau).$$ 
Therefore, for every $(B,A)\in Q$ it lies strictly between $\frac1{\omega}(-\sin\phi-1)$ and $\frac1{\omega}(-\sin\phi+1)$. 
This implies the above-mentioned inequalities at $\phi=\pm\frac{\pi}2$ and proves the proposition.
\end{proof}

\begin{example} In the case, when $A=0$, the differential equation defined by vector field (\ref{josvect}) takes the form 
$$\frac{d\phi}{d\tau}=\frac1{\omega}(B-\sin\phi).$$
This is an autonomous differential equation  that can be solved explicitly, see, e.g., \cite[p. 126, formula (6.2.6)]{bar}. 
In the case, when $|B|\leq1$, its right-hand side vanishes at the 
points, where $\sin \phi=B$, and these are fixed points for the flow maps, and hence,  the rotation number equals zero. Let now 
$B>1$. Then the right-hand side is positive everywhere, hence  the rotation number is positive. 
Let $u_1$ and $u_2$ be the roots of the quadratic polynomial $B(1+u^2)-2u=0$: 
$$u_{1,2}=\frac{1\pm i\sqrt{B^2-1}}B=e^{\pm i\beta}, \ \beta=\arccos\frac1B.$$ 
 The general solution of the 
above equation is given by the implicit function 
$$\tau=\omega\int_{\phi_0}^{\phi}\frac{d\phi}{B-\sin\phi}=-i\alpha^{-1}\ln(\frac{u-u_1}{u-u_2})+c, \ 
u=\tan\frac{\phi}2, $$
$$\alpha=-i\omega^{-1}\frac{u_1-u_2}{u_1+u_2}=\omega^{-1}\sqrt{B^2-1}, \ c\equiv const.$$
Then we get 
$$\tan\frac{\phi}2=\frac{\exp(i\alpha(\tau-c))u_2-u_1}{\exp(i\alpha(\tau-c))-1}=
\frac{\sin(\frac{\alpha}2(\tau-c)-\beta)}{\sin(\frac{\alpha}2(\tau-c))},\  c\equiv const.$$
\end{example} 

For $l\in\nn$ let $\mathcal P_l\in\{ B=l\omega\}$ be the point  with the maximal $A-$coordinate, for which  the corresponding 
Heun equation (\ref{heun2*})  has a   polynomial solution. 

\begin{conjecture} All $\mathcal P_l$ lie on the same line, see Fig. 7. 
\end{conjecture}

\begin{figure}[ht]
  \begin{center}
   \epsfig{file=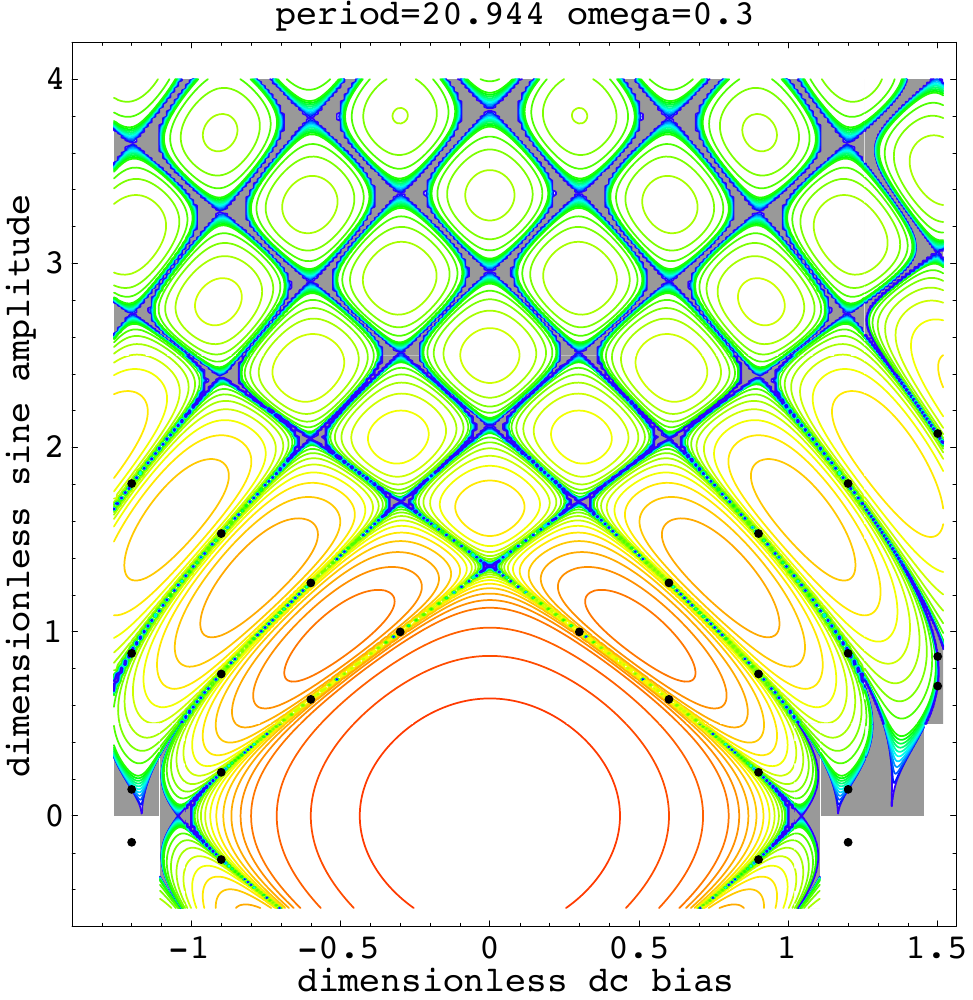, width=0.5\textwidth}
   \caption{Phase-lock areas for $\omega=0.3$; the marked points correspond to Heun equations (\ref{heun2*}) with polynomial solutions. They are described by Theorem \ref{poteig}.} 
  \end{center}
\end{figure}

\begin{conjecture} As $\omega\to0$,  for every $r$  the set $L_{r,1}:=L_r\cap\{ A\geq A(\mathcal A_{r,1})\}$ 
tends to the ray  
$\{ A\geq 1\}$ in the $A$-axis. Namely, the maximal distance of a point of the set $L_{r,1}$ to the ray $\{ A\geq1\}$ 
tends to zero.
\end{conjecture}

In what follows we will discuss in detail the next two conjectures that are closely related to Conjectures \ref{garl} and \ref{interv}. 

\begin{conjecture} \label{c1} For $r\in\nn$ the phase-lock area with rotation number $r+1$ does not intersect the line  
$$\Lambda_r=\{ B=\omega r\}\subset\rr^2.$$
\end{conjecture}

\begin{conjecture} \label{c2} For $r\in\nn$ the phase-lock area with rotation number $r$ does not intersect the line $\Lambda_{r-2}$. 
\end{conjecture}

\begin{remark} \label{graph} Conjecture \ref{c1}  implies Conjecture \ref{c2}. Indeed, the points $(B,A)$ with $A>0$ large enough 
of the phase-lock area $L_r$,  $r\in\zz$ 
lie close to $\Lambda_r$, i.e., they are separated from the line $\Lambda_{r-2}$ by $\Lambda_{r-1}$. This follows from the fact 
that its boundary consists of graphs of two functions $B=g_{\pm r}(A)$ and $g_{\pm r}(A)\to r\omega$, 
as $A\to+\infty$ (follows from results of \cite{RK}). Therefore, if the phase-lock area $L_r$ does not intersect the line $\Lambda_{r-1}$, 
then it also does not intersect $\Lambda_{r-2}$. 
Each one of Conjectures \ref{c1}, \ref{c2} together with \cite[theorems 1.2, 3.17]{4} (see Remark \ref{nwrem}) 
imply Conjecture \ref{garl}. 
\end{remark}

{\bf A possible strategy for Conjecture \ref{c1}.} If the boundary of the phase-lock area with rotation number $r+1$ intersects the line 
 $\Lambda_r$, 
then the intersection points correspond to parabolic monodromy operator of Jordan cell type with both eigenvalues equal to -1 
(Proposition \ref{newex}). That is, some of equations (\ref{r0eql}) or (\ref{r1eql}) should hold 
at each intersection point. 

\begin{conjecture} \label{c3} Let $l\in\nn$, and let the parameter $\mu$  satisfy some of equations  (\ref{r0eql}) if $l\notin 2\zz$, or (\ref{r1eql}) if $l\in 2\zz$. Consider
the  corresponding monodromy eigenfunction of Heun equation (\ref{heun}) from Proposition \ref{newex} with eigenvalue $-1$. 
Then the corresponding solution of the Riccati equation from the beginning of Subsection 5.3
 gives  a periodic solution of the corresponding equation (\ref{josvect}) on two-torus
 having rotation number between 0 and $l$. 
\end{conjecture}

Conjecture \ref{c3} would imply Conjecture \ref{c1}. 

\medskip

{\bf A possible strategy for Conjecture \ref{c2}.} We know that for $\omega\geq1$ the statements of Conjecture \ref{c1} and hence Conjecture \ref{c2} 
hold (Chaplygin Theorem argument, see \cite[lemma 4]{buch1} and \cite[proposition 3.4]{4}). The adjacencies of a phase-lock area with rotation number $\rho$ cannot lie on lines  $\Lambda_l$ with $l\not\equiv\rho(mod2\zz)$, see \cite[theorem 3.17]{4}; this also follows from 
Proposition \ref{newex}. 
 Suppose that for a certain ``critical'' value $\omega=\omega_0<1$  the boundary of the phase-lock area number $l+2>0$  moves from the right to the left,  as $\omega$ decreases to $\omega_0$,  and touches the line $\Lambda_l$ at some point $(B,A)$, 
as $\omega=\omega_0$. Then there are two possibilities for the corresponding Heun equation:

- the associated Heun equation (\ref{heun2}) (equation (\ref{heun}) with $l$ replaced by $-l$) 
has a polynomial solution. But this case is forbidden by Buchstaber--Tertychnyi 
result \cite[theorem 4]{bt0}, which states that then the corresponding rotation number cannot be greater than $l$. 

- the point $(B,A)$ an adjacency:  
Heun equation (\ref{heun}) has a solution holomorphic on $\cc$. This together with the above-mentioned known fact that the 
boundaries of phase-lock areas   are graphs of functions (Remark \ref{graph}) implies that both boundary components of 
the phase-lock area with rotation number $l+2$ are tangent to the line $\Lambda_l$ at the point $(B,A)$.

\begin{conjecture} \label{c4} For every $\omega>0$ for every adjacency $(B_0,A_0)\in\rr_+^2$ of any phase-lock area the branches of its boundary at $(B_0,A_0)$ cannot be both 
tangent to the vertical line $\{ B=B_0\}$. 
\end{conjecture}

\begin{proposition} \label{prconj2} Conjecture \ref{c4} implies Conjecture \ref{c2}, and hence, \ref{garl}. 
\end{proposition}

The proposition follows from the above argument and Remark \ref{graph}. 

A possible approach to Conjecture \ref{c4} could be studying 
equations (\ref{r0eql}) and (\ref{r1eql}) defining the boundaries and to see what happens with them 
when the ``non-resonant'' parameters approach the resonant ones. A first step is done below.

\subsection{Description of  boundaries of phase-lock areas near adjacencies. Relation to Conjecture \ref{c2}}
Let us write down equation (\ref{r0eql}) on the boundaries in a neighborhood of a line $\Lambda_{l_0}$, $l_0\in2\zz$. 
Let us recall the formulas for the corresponding matrices:
$$
M_k=M_k(\la,\mu,l)=\left(\begin{matrix}  1+\frac{\la}{k^2-\frac{l^2}4} & \frac{\mu^2}{k^2-\frac{l^2}4}\\  1 & 0\end{matrix}
\right), \ R_k=M_kM_{k+1}\dots.$$
Equation (\ref{r0eql})  for the boundaries  is  
$$ R_{0,21}\pm\omega l(R_{0,21}-R_{0,11})=0.$$
Note that the matrices $M_k$ are analytic in a neighborhood of the hyperplane $\{ l=l_0\}$ except for the matrix $M_{\frac{l_0}2}$, which has pole of order one along the latter hyperplane. One has 
$$\frac{l_0^2-l^2}4M_{\frac{l_0}2}=\left(\left(\begin{matrix} & \la & \mu^2\\ & 0 & 0\end{matrix}\right)+\frac{l_0^2-l^2}4\left(\begin{matrix} & 1 & 0\\ & 1 & 0
\end{matrix}\right)\right).$$

Set 
$$\mathcal R=\frac{l_0^2-l^2}4R_0, \ X=M_0\dots M_{\frac{l_0}2-1}, \ X=Id \text{ for } l_0=0.$$
 One has 
 $$\mathcal R=X\left(\left(\begin{matrix}  \la & \mu^2\\  0 & 0\end{matrix}\right)+
 \frac{l_0^2-l^2}4\left(\begin{matrix}  1 & 0\\  1 & 0
\end{matrix}\right)\right)R_{\frac{l_0}2+1},$$
$$\left(\begin{matrix}  \la & \mu^2\\  0 & 0\end{matrix}\right)R_{\frac{l_0}2+1}|_{l=l_0}=\left(\begin{matrix}  \xi_{l_0}(\la,\mu) & 0\\
 0 & 0\end{matrix}\right),$$
by (\ref{xil}), the equality $R_{s,12}=R_{s,22}=0$ for every $s$ (Corollary \ref{clem}) and since the matrices $M_{\frac{l_0}{2}+k}(\la,\mu,l_0)$, $R_{\frac{l_0}{2}+k}(\la,\mu,l_0)$ coincide with the matrices 
$M_k$, $R_k$ preceding (\ref{xil}) with $l=l_0$.
Therefore, 
\begin{equation}\mcr=\xi_{l_0}(\la,\mu)\left(\begin{matrix}  X_{11} & 0 \\  X_{21} & 0\end{matrix}\right)+(l-l_0)\chi(l,\la,\mu),
\label{rzeta}\end{equation}
where $\chi(l,\la,\mu)$ is a holomorphic matrix-valued function on a neighborhood of the hyperplane $\{ l=l_0\}$.
Now  equation (\ref{r0eql}) can be rewritten as 
\begin{equation}\mcr_{21}\pm\omega l(\mcr_{21}-\mcr_{11})=0.\label{r0l0}\end{equation}
Taking into account asymptotics (\ref{rzeta}) one gets asymptotic form of equation (\ref{r0l0}): 
\begin{equation}\xi_{l_0}(\la,\mu)(X_{21}\pm\omega l(X_{21}-X_{11}))+(l-l_0)(\chi_{21}\pm\omega l(\chi_{21}-\chi_{11}))=0.\label{r0as}\end{equation}

Now let us consider the case, when $l_0\in 2\zz+1$, and let us write down equation (\ref{r1eql}) in a neighborhood of the line 
$\Lambda_{l_0}$. The  corresponding matrices from (\ref{mkrk1}) are 
$$M_k=\left(\begin{matrix}  1+\frac{\la}{(k-\frac12)^2-\frac{l^2}4} & \frac{\mu^2}{(k-\frac12)^2-\frac{l^2}4} \\  1 & 0\end{matrix}\right), \ 
R_k=M_kM_{k+1}\dots.$$
Set 
$$\mcr=\frac{l_0^2-l^2}4R_1, \ X=M_1\dots M_{\frac{l_0+1}2-1}; \ X=Id \text{ for } l_0=1.$$
Analogously to the above calculations, we get asymptotic relation (\ref{rzeta}). Together with (\ref{r1eql}), it implies
\begin{equation}\xi_{l_0}(\la,\mu)(X_{11}\pm2\omega \mu(X_{11}-X_{21}))+
(l-l_0)(\chi_{11}\pm2\omega \mu(\chi_{11}-\chi_{21}))=0,\label{r1as}\end{equation}
where $\chi(l,\la,\mu)$ is a holomorphic matrix-valued 
function on a  neighborhood of the hyperplane $\{ l=l_0\}$. Set 
$$\zeta_l(\omega,\mu)=\xi_l(\la,\mu)=\xi_l(\frac1{4\omega^2}-\mu^2,\mu).$$

\begin{conjecture} \label{c5} For every $\omega>0$ and $l\in\zz$, $l\geq0$ the zeros of the function $\zeta_l$ are simple, that is, 
 $\frac{\partial\zeta_l}{\partial\mu}\neq0$ {\it at zeros of the function $\zeta_l$.}
\end{conjecture}

\begin{remark} The matrices $X$ in both cases treated above are non-degenerate for $l=l_0$.  
This implies that in formulas (\ref{r0as}) and (\ref{r1as}) the multiplier at $\xi_{l_0}$ is non-zero for at least one choice of sign. 
\end{remark}

Conjecture \ref{c5} together with the above remark 
would imply that {\it for every $l\in\zz$ at every adjacency in the line $\Lambda_l$  at least one  boundary component  of the 
corresponding phase-lock area (depending on the above-chosen sign) 
is transversal to the line $\Lambda_l$.}

\begin{proposition} Conjecture \ref{c5} implies Conjectures \ref{c4},  \ref{c2}, \ref{garl}.
\end{proposition}

\begin{proof} 
Conjecture \ref{c5} implies  that no adjacency can be born 
at a tangency of both boundary components 
with a line $l=l_0$, $l_0\in\zz$, by transversality of one of them to the latter line
(the above statement). In other words, it implies Conjecture \ref{c4}, and hence,  
Conjectures \ref{c2}, \ref{garl}, by Proposition \ref{prconj2}. 
\end{proof}

{\bf Open Question.} 
Study the degeneracy of equation (\ref{pasterho}) on non-integer level curves of rotation number, as the latter 
number tends to an integer value. The level curves should tend  to  boundaries of  phase-lock areas. 
How to retrieve equations (\ref{r0eql}), (\ref{r1eql}), (\ref{r0as}) and (\ref{r1as}) on the boundaries and equation $\xi_l(\la,\mu)=0$ on the adjacencies 
from asymptotics of degenerating equation (\ref{pasterho}), as $r$ tends to an integer number? Find a method to calculate the 
coordinates of the adjacencies corresponding to an integer rotation number $r_0$ via the level curves $\{\rho=r\}$, as $r\to r_0$. 

\subsection{Asymptotic behavior of phase-lock areas for small $\omega$} 
The problem to describe the behavior of the phase-lock areas, as $\omega\to0$, is known and motivated by physical 
applications. Concerning this problem V.M.Buchstaber and S.I.Tertychnyi, and later D.A.Filimonov, V.A.Kleptsyn and 
I.V.Schurov  have done numerical experiments. These experiments  have shown that after appropriate rescaling of the variables $(B,A)$, 
   the phase-lock areas tend to  open sets (which we will call the {\it limit rescaled phase-lock areas}) whose components  form a 
   partition of the whole plane similar  to a chess table turned by $\frac{\pi}4$, see Fig. 5 for $\omega=0.3$.

{\bf Open Question 7.} Obtain theoretical results on the problem described above. 


\section{Acknowledgements}

The authors are grateful to S.I.Tertychnyi for helpful discussions and to R.R.Gontsov and A.V.Klimenko for careful reading the paper and helpful remarks. The authors are grateful to R.R.Gontsov for translating the paper into Russian.

\end{document}